\documentclass[a4paper]{article}
\usepackage{geometry}
\usepackage[utf8]{inputenc}
\usepackage{amsmath,amsfonts,amsthm}
\usepackage{hyperref}
\usepackage{comment}
\usepackage{paralist}
\usepackage[normalem]{ulem}

\title{Multiplicative structures and random walks in o-minimal groups}
\author{Hunter Spink\thanks{Department of Mathematics, Stanford University, Stanford, CA 94305.
Email: \href{mailto:hspink@stanford.edu} {\nolinkurl{hspink@stanford.edu}}.}}
\date{}

\usepackage{amsthm}
\usepackage[nameinlink,capitalise,noabbrev]{cleveref}
\newtheorem{thm}{Theorem}[section]
\crefname{thm}{Theorem}{Theorems}
\newtheorem{clm}[thm]{Claim}
\newtheorem{prop}[thm]{Proposition}

\newtheorem{cor}[thm]{Corollary}
\newtheorem{qu}[thm]{Question}

\newtheorem{conjecture}[thm]{Conjecture}
\crefname{conjecture}{Conjecture}{Conjectures}
\newtheorem{lem}[thm]{Lemma}
\crefname{lem}{Lemma}{Lemmas}
\newtheorem{fact}[thm]{Fact}

\newtheorem{thmx}{Theorem}

\theoremstyle{definition}
\newtheorem{defn}[thm]{Definition}
\newtheorem{exmp}[thm]{Example}
\newtheorem{notat}[thm]{Notation}
\newtheorem{rem}[thm]{Remark}

\usepackage{color}
\usepackage{verbatim}
\usepackage{paralist}

\begin{document}

\global\long\def\RR{\mathbb{R}}%
\global\long\def\QQ{\mathbb{Q}}%
\global\long\def\E{\mathbb{E}}%
\global\long\def\Var{\operatorname{Var}}%
\global\long\def\CC{\mathbb{C}}%
\global\long\def\NN{\mathbb{N}}%
\global\long\def\ZZ{\mathbb{Z}}%
\global\long\def\Bad{\operatorname{Bad}}%
\global\long\def\Ber{\operatorname{Bernoulli}}%
\global\long\def\Inf{\operatorname{Inf}}%
\global\long\def\vol{\operatorname{vol}}%
\global\long\def\conv{\operatorname{conv}}%
\global\long\def\floor#1{\left\lfloor #1\right\rfloor }%
\global\long\def\ceil#1{\left\lceil #1\right\rceil }%

\global\long\def\dimrob{\operatorname{dim}_{\mathrm{rob}}}%
\global\long\def\dimself{\operatorname{dim}_{\mathrm{self}}}%

\let\originalleft\left
\let\originalright\right
\renewcommand{\left}{\mathopen{}\mathclose\bgroup\originalleft}
\renewcommand{\right}{\aftergroup\egroup\originalright}

\let\OLDthebibliography\thebibliography
\renewcommand\thebibliography[1]{
  \OLDthebibliography{#1}
  \setlength{\parskip}{0pt}
  \setlength{\itemsep}{3pt plus 0.3ex}
}

\global\long\def\mk#1{\textcolor{red}{\textbf{[MK comments:} #1\textbf{]}}}

\global\long\def\hs#1{\textcolor{red}{\textbf{[HS comments:} #1\textbf{]}}}

\newcommand{\hunter}[1]{\textcolor{red}{#1}}

\maketitle

\begin{abstract}
We prove structure theorems for o-minimal definable subsets $S\subset G$ of definable groups containing large multiplicative structures, and show definable groups do not have bounded torsion arbitrarily close to the identity. As an application, for certain models of $n$-step random walks $X$ in $G$ we show upper bounds $\mathbb{P}(X\in S)\le n^{-C}$ and a structure theorem for the steps of $X$ when $\mathbb{P}(X\in S)\ge n^{-C'}$.

\end{abstract}
\section{Introduction}
\label{sec:intro}
For a sufficiently nice subset $S$ of a topological group $G$, such as a submanifold of a Lie group, we investigate settings where the ``dimension'' of $S$ obstructs which finite multiplicative structures can appear in $S$. In these settings, we are able to study questions of finite additive combinatorics for such potentially infinite subsets $S\subset G$, where $G$ is considered as an abstract group.

If the subsets $S$ under consideration are general enough to interpolate through arbitrary countable subsets of $G$, then nothing can be said. A natural context which restricts arbitrary interpolation is definability in an o-minimal structure $\mathbb{M}_{\mathcal{F}}:=(M,<,\mathcal{F})$ extending a dense linear order $(M,<)$. In recent years, there have been a number of striking applications of o-minimality, particularly in the study of Diophantine properties of special points on varieties (see for example \cite{P11,PU08}), and counting rational and algebraic points of bounded height on transcendental sets (see for example \cite{PW06,P09}). We collect definitions concerning o-minimality in \Cref{sec:omindefs} and facts about o-minimality in \Cref{sec:dim}, and  refer the interested reader to the introductory texts of van den Dries \cite{vdD98} and Coste \cite{Cos00}.

Informally, one of our main results is that for $G$ a Lie group and $S\subset G$ an analytic submanifold, both definable in the restricted-analytic exponential structure $\mathbb{R}_{an,\exp}$ (see \Cref{exmp:Ranexp}), then if $S$ contains finite multiplicative substructures of ``dimension'' (appropriately construed, see \Cref{defn:dimrob}) exceeding the Hausdorff dimension $\dim(S)$, then  $S$ must contain an ``exponential arc'' of $G$ (\Cref{defn:exparc}). Hence such definable subsets $S\subset G$ avoiding ``exponential arcs'' have interesting finite additive combinatorics.

As a concrete corollary, if $S$ contains unboundedly long arithmetic progressions $\{h,gh,\ldots,g^{m-1}h\}$, then we can find unboundedly long arithmetic progressions with arbitrarily small step size starting at a fixed $s\in S$ -- we will show this holds in any o-minimal structure, even when $G$ is not a Lie group. We note that $\mathbb{R}_{an,\exp}$ excludes certain types of oscillatory behaviour such as $S=\{(x,y)\in \mathbb{R}^2:y=\sin(x)\}$, which contains the arithmetic progression $2\pi\mathbb{Z}\times \{0\}$ but not an entire interval (an exponential arc in $G=\mathbb{R}^2$).

Formalizing this, we will define two notions of combinatorial dimension $\dimrob(S)$ and $\dimself(S)$ for subsets $S\subset G$, based on the multiplicative substructures contained in $S$. We show that for $S$ definable in an o-minimal structure not containing an exponential arc in the $\mathbb{R}_{an,exp}$ setting, or unboundedly large arithmetic progressions in the more general $\mathbb{M}_{\mathcal{F}}$ setting, these  combinatorial dimensions are finite, and bounded above by the o-minimal dimension $\dim(S)$ (the Hausdorff dimension when $G$ is a Lie group).

The main application of our results (see \Cref{sec:probapp}) will be to show that questions concerning the maximum concentration probability $\max_{g\in G}\mathbb{P}(X=g)$ for a random walk $X$ obtained as a product of independent random elements of $G$ (possibly taken from different distributions), as considered in Littlewood-Offord theory, behave remarkably similar to analogous questions for $\mathbb{P}(X\in S)$. This builds on previous work of Fox, Kwan, and the present author \cite{fox2021geometric} for subsets $S$ of the topological group $(\mathbb{R}^d,+)$, which are definable in an o-minimal structure extending $\mathbb{R}_{alg}:=(\mathbb{R},<,+,\times)$.

We start by defining $\dimrob(S)$.

\begin{defn}\label{defn:dimrob}
Define the \emph{$r$-uniform multiplication hypergraph of $S$} to be
$$\mathcal{H}_{S,r}:=\{(g_1,\ldots,g_r)\in G^r: \prod_{i=1}^r g_i\in S\}\subset G^r,$$
and say that $S$ \emph{robustly contains $r$-dimensional multiplicative structure} if $\mathcal{H}_{S,r}$ contains grids $A_1\times \cdots \times A_r$ with $|A_i|=C$ for all $C$. Define the \emph{robust multiplicative dimension} of $S$ as
$$\dimrob(S):=\sup\{r:S\text{ robustly contains $r$-dimensional multiplicative structure}\}\in \mathbb{N}\cup \{\infty\}.$$
\end{defn}
Note that in the special case that $A_1,\ldots,A_r$ are arithmetic progressions of the form $\{h,gh,\ldots,g^{C-1}h\}$ for various $g,h\in G$, then the pointwise product set $A_1\cdots A_r$ is a ``generalized rank $r$ arithmetic progression'', the discrete analogue of a parallelepiped or zonotope.

Our first theorems center around the following question.
\begin{qu}\label{qu:rob}
Can we geometrically characterize subsets $S\subset G$ which have $\dimrob(S)=r$?
\end{qu}

We recall as a starting point a very general result of Chernikov, Galvin, and Starchenko on definable r-partite hypergraphs in an o-minimal structure containing arbitrarily large grids.
\begin{thm}[Chernikov, Galvin, Starchenko {\cite[Theorem 5.12]{CGS20}}]
If $\mathcal{H}\subset V_1\times \ldots \times V_r$ is a definable subset of a product of definable sets in an o-minimal structure $\mathbb{M}_{\mathcal{F}}$, and $\mathcal{H}$ contains grids $A_1\times \cdots \times A_r$ with $|A_i|=C$ for all $C$, then there are definable curves $\mathcal{C}_1,\ldots,\mathcal{C}_r$ with $\mathcal{C}_i\subset V_i$ such that $\mathcal{H}$ contains $\mathcal{C}_1\times \cdots \times \mathcal{C}_r$.
\end{thm}
Applying the theorem to the hypergraph $\mathcal{H}_{S,r}$ yields the following.
\begin{cor}\label{cor:Chernikov}
Let $S\subset G$ be a definable subset of a definable group $G$ in an o-minimal structure $\mathbb{M}_{\mathcal{F}}$. If $\dimrob(S)<\infty$, then $\dimrob(S)$ is the largest $r$ such that there are definable curves $\mathcal{C}_1,\ldots,\mathcal{C}_r\subset G$ such that $S$ contains the pointwise product $\mathcal{C}_1\cdots \mathcal{C}_r$. If $\dimrob(S)=\infty$ then $S$ contains such products for all $r$.
\end{cor}

Refining this, our first two results give a simpler characterization of those $S$ such that $\dimrob(S)=\infty$, and show that when $\dimrob(S)<\infty$, we may assume the multiplication map $\mathcal{C}_1\times \cdots \times \mathcal{C}_r\to \mathcal{C}_1\cdots \mathcal{C}_r$ is injective. This latter result in particular implies $\dimrob(S)\le \dim(S)$ when $\dimrob(S)<\infty$, so in this case the $o$-minimal dimension controls the finite multiplicative structures that can appear in $S$.

Our first theorem considers the special case that $S\subset G$ is a definable subset of a definable group in the restricted analytic-exponential o-minimal structure $\mathbb{R}_{an,\exp}$ of van den Dries and Miller \cite{vdD98} (which in particular implies $G$ canonically has the structure of a Lie group and $S$ is a finite union of submanifolds).

\begin{exmp}\label{exmp:Ranexp}
We will describe later precisely what it means for $S\subset G$ to be definable in $\mathbb{R}_{an,\exp}$, but note that this includes as special cases when
\begin{itemize}
    \item $G$ is any semi-algebraic group, such as $GL_d(\mathbb{R})$, $GL_d(\mathbb{C})$, the unitary group $U_d(\mathbb{C})$, $(\mathbb{R}^d,+)$ or an abelian variety, and
    \item $S\subset G$ is any semi-algebraic subset, or more generally any Boolean combination of subsets of $G$ defined by equalities and inequalities of functions which are compositions of real constants, $+,\times,\exp,\log$, and ``restricted analytic functions'' $f|_{[0,1]^\ell}$ where $f:N\to \mathbb{R}$ is an analytic function on an open neighborhood $N$ of $[0,1]^\ell\subset \mathbb{R}^\ell$ for various $\ell$.
\end{itemize} 
\end{exmp}

In this setting, even without arbitary interpolation there is a natural supply of curves $S$ with $\dimrob(S)=\infty$ which arise from the Lie group exponential map.
\begin{defn}\label{defn:exparc}
For a Lie group $G$ with lie algebra $\mathfrak{g}$, define an \emph{exponential arc} to be a subset of the form
$$\{\exp_{G}(tM)h:t\in [0,c]\}\subset G,$$
where $\exp_{G}$ is the Lie group exponentiation map from a neighborhood $0\in B\subset \mathfrak{g}$ to $G$, $0\ne M \in c^{-1}B$, $h\in G$ and $0<c\in \mathbb{R}$.
\end{defn}
For any such arc, if we take values of $t$ lying in a sum $A_1'+\cdots+A_r'$ with $A_i'\subset [0,\frac{c}{r}]$, then this corresponds to a grid $A_1\times \cdots \times A_r\subset \mathcal{H}_{S,r}$ with $|A_i|=|A_i'|$ for all $i$. Hence an exponential arc has $\dimrob(S)=\infty$.

Before stating our first theorem, we recall that a \emph{definable family} of subsets of $T$ is the collection of fibers over a definable set $\mathcal{F}$ (the indexing set) of a definable subset $\widetilde{S}\subset T\times \mathcal{F}$, considered as subsets of $T$. For example the family of ellipses $\{(x,y):ax^2+by^2=1\}_{a,b\in \mathbb{R}_{>0}}$ is a definable family of subsets of $\mathbb{R}^2$ in the structure $\mathbb{R}_{an,\exp}$, corresponding to $\{(x,y,a,b):ax^2+by^2=1\}\subset \mathbb{R}^2\times \mathbb{R}_{>0}^2$.

\begin{thmx}
\label{thm:trajectory}
Suppose that $S\subset G$ is a definable subset of a definable group in  $\mathbb{R}_{an,\exp}$. Then
\begin{enumerate}[a)]
    \item  $\dimrob(S)=\infty$ exactly when $S$ contains an exponential arc. Furthermore, in a definable family of subsets of $G$, there exists an $N$ such that $S$ contains an exponential arc if and only if $S$ contains an arithmetic progression of length $\ge N$.
    \item If $S$ does not contain an exponential arc, then $\dimrob(S)= r$ is the largest number so that there is an analytic injection $\Gamma(t_1,\ldots,t_r):[0,1]^r\to S$ of the form $\Gamma(t_1,\ldots,t_r)=\gamma_1(t_1)\cdots \gamma_r(t_r)$ where each $\gamma_i(t):[0,1]\to G$ is an analytic injection. In particular, $\dimrob(S)\le \dim(S)$.
\end{enumerate}
\end{thmx}
A major obstacle to proving this is the non-definability of $\exp_G$ even for $G=GL_d(\mathbb{R})$ when $d\ge 2$. For example, if $\exp_{GL_2(\mathbb{R})}$ were a definable function, then so would the entries of $$\exp_{GL_2(\mathbb{R})}(t\begin{bmatrix}0 & -1\\1 & 0\end{bmatrix})=\begin{bmatrix}\cos(t) & -\sin(t)\\\sin(t) & \cos(t)\end{bmatrix},$$ and although $\cos(t)$ and $\sin(t)$ are analytic, they are not in the restricted-analytic exponential structure, or in fact any o-minimal structure since they can be used to define infinite discrete subsets of $\mathbb{R}$ (which will be explicitly forbidden in the definition of o-minimality). This issue arises in a number of contexts in o-minimal geometry, for example in the development of o-minimal complex analysis by Peterzil and Starchenko \cite{compl1,compl2,compl3}, and in the study of the dynamics of real analytic functions by Scanlon \cite{Scanlon}. It is known that o-minimal structures augmented with the indicator functions of certain unbounded exponential arcs satisfy tameness properties slightly weaker than o-minimality, see Miller \cite{Miller}.

Even if $G$ is a more general topological group without a notion of exponential arc, it is easy to see that containing unboundedly large arithmetic progressions $\{h,gh,\ldots,g^{m-1}h\}$ still causes $\dimrob(S)=\infty$.
\begin{defn}
Say that a subset $S\subset G$ contains \emph{unboundedly large arithmetic progressions} if for all $m$ there is an $m$-element subset $\{h,gh,\ldots,g^{m-1}h\}\subset S$ (in particular $g$ is not torsion of order $<m$).
\end{defn}
Our second theorem considers $S\subset G$ a definable subset of a definable group in an arbitary o-minimal structure $\mathbb{M}_{\mathcal{F}}:=(M,<,\mathcal{F})$ extending a dense linear order $(M,<)$. In an o-minimal structure there is a well-defined notion of ``dimension'' (which agrees with Hausdorff dimension when $M=\mathbb{R}$), and every definable group is canonically equipped with a topology called the ``t-topology'' which makes it a topological group. Since the study of definable groups was first initiated by Pillay \cite{Pillay}, their properties and classification have become a cornerstone of model theory. We refer the reader to the surveys of Otero \cite{otero_2008} and, for o-minimal structures extending a real-closed field, Conversano \cite{conversano2020groups}.
\begin{thmx}
\label{thm:generalominimal}
Suppose that $S\subset G$ is a definable subset of a definable group in an o-minimal structure $\mathbb{M}_{\mathcal{F}}$. Then
\begin{enumerate}[a)]
    \item $\dimrob(S)=\infty$ exactly when $S$ contains unboundedly large arithmetic progressions, in which case there is an $s\in S$ such that every open neighborhood of $s$ contains unboundedly large arithmetic progressions starting at $s$. Furthermore, in a definable family of subsets of $G$, there exists an $N$ such that $S$ contains unboundedly large arithmetic progressions if and only if $S$ contains an arithmetic progression of length $\ge N$.
    \item If $S$ does not contain unboundedly large arithmetic progressions, then $\dimrob(S)= r$ is the largest number so that there is a continuous definable injection $\Gamma(t_1,\ldots,t_r):\prod_{i=1}^r [a_i,b_i]\to S$ of the form $\Gamma(t_1,\ldots,t_r)=\gamma_1(t_1)\ldots \gamma_r(t_r)$ where  each $\gamma_i(t):[a_i,b_i]\to G$ is a continuous definable injection. In particular, $\dimrob(S)\le \dim(S)$.
\end{enumerate}
\end{thmx}
The following question would be interesting to answer.
\begin{qu} Is there an analogue of an ``exponential arc'' for $G$ definable in an arbitrary o-minimal structure which a definable subset has to contain if it contains unboundedly large arithmetic progressions?
\end{qu}

A key structural result for the torsion in definable groups we need to show along the way may be of independent interest. Recall that Gleason \cite{Gleason} and Yamabe \cite{Yamabe} established the following property for locally Euclidean topological groups, a key component to the resolution of Hilbert's fifth problem (that every locally Euclidean topological group is a Lie group) by Gleason \cite{Gleason} and Montgomery-Zippin \cite{MontZip}:
\begin{align}\label{tag:NSS}\tag{NSS}
\text{There is a neighborhood $\widehat{N}$ of $\operatorname{id}_G$ such that $\widehat{N}$ contains no non-trivial subgroups of $G$.}
\end{align}
The (NSS) property was recently identified by Hrushovskii \cite{Hrushovski} as being important in a model-theoretic context, which led to a qualitative theory of approximate groups (later refined in a quantitative form by Breuillard, Green, and Tao \cite{ApproximateGroup}).

There are examples of definable groups $G$ that do not satisfy the (NSS) property, for example $G$ the additive group underlying a sufficiently saturated elementary extension of $(\mathbb{R},<,+,\times)$. Consider the following weakening of (NSS) which we call the ``no small torsion'' property:
\begin{align}\label{tag:NST}\tag{NST}
    \text{For every $m$, there is a neighborhood $\operatorname{id}_G\in \widehat{N}_m$ with no element of order $\le m$ in $\widehat{N}_m\setminus \{\operatorname{id}_G\}$.}
\end{align}
Our third theorem shows definable groups always have the property \eqref{tag:NST}.
\begin{thmx}
\label{thm:NST}
If $G$ is a definable group in an o-minimal structure $\mathbb{M}_{\mathcal{F}}$, then $G$ satisfies the ``no small torsion'' property \eqref{tag:NST}.
\end{thmx}
 
In \cite{fox2021geometric}, Fox, Kwan and the present author considered a more refined notion of dimension than $\dimrob(S)$ for applications to anti-concentration probabilities for random walks in $(\mathbb{R}^d,+)$ on definable sets. We repeat the definition in the context of general groups.
\begin{defn}
\label{defn:selftranslate}
Given a group $G$, we define a property of subsets $S\subset G$ called \emph{self-translate dimension $\le k$ of complexity $C$} recursively as follows.
\begin{itemize}
    \item All subsets $S\subset G$ with $|S| \le C$ have self-translate dimension $\le 0$ of complexity $C$.
    \item $S$ has self-translate dimension $\le k$ of complexity $C$ if we can write $S=S_1\cup \ldots \cup S_C$ such that for all $i$ and $\operatorname{id}_G\ne g\in G$ we have $S_i \cap gS_i$ has self-translate dimension $\le k-1$ of complexity $C$.
\end{itemize}
Then define the \emph{self-translate dimension of $S$} as
$$\dimself(S)=\inf_k\{k:\text{$S$ has self-translate dimension $\le k$ for some complexity $C$}\}\in \mathbb{N}\cup \{\infty\}.$$
\end{defn}
The motivating example (as discussed in \cite{fox2021geometric}) is when $G=(\mathbb{C}^d,+)$ and $S\subset \mathbb{C}^d$ is an algebraic variety (i.e. defined by a system of polynomial equations) not containing a line segment. Then one can inductively show that $S$ is of self-translate dimension $\le \dim(S)$ of complexity $O_{\deg(S)}(1)$ by taking $S_1\cup \ldots \cup S_C$ to be the decomposition into irreducible components. Indeed, for $v\ne 0$ irreducibility implies $S_i\cap (S_i+v)$ has dimension $\le \dim(S)-1$ (as one can show this is a proper subvariety of $S_i$ by the line segment condition), and B\'ezout's theorem guarantees $\deg(S_i \cap (S_i+v))=O_{\deg(S)}(1)$.

\begin{rem}A set $S$ has self-translate dimension $\le k$ of complexity $1$ precisely if $S$ contains no product set of the form $\{a_1,b_1\}\cdot \{a_2,b_2\}\cdot \ldots \cdot \{a_{k+1},b_{k+1}\}$ with $a_i\ne b_i$ for all $i$, or equivalently no product set $a\{\operatorname{id}_G,x_1\}\ldots\{\operatorname{id}_G,x_n\}$ with all $x_i \ne \operatorname{id}_G$. For $G=(\mathbb{R}^d,+)$, it was shown in \cite[Proposition 6.2]{fox2021geometric} that if $S$ is the boundary of a strictly convex set, then it can be partitioned into $\le 2d$ sets each of self-translate dimension $\le d-1$ of complexity $1$.
\end{rem}

As implicitly observed in \cite{fox2021geometric}, the concrete advantage of $\dimself(S)$ over $\dimrob(S)$ is that if $|A_1|=\ldots=|A_{r+1}|=n$, then
$$\dimself(S)\le r \implies |\mathcal{H}_{S,r}\cap A_1\times \cdots \times A_{r+1}|=O_{S,r}(n^{r+1-1/2^{\dimself(S)}}),$$
while standard results on hypergraph avoidance only imply $$\dimrob(S)\le r \implies |\mathcal{H}_{S,r}\cap A_1\times \cdots \times A_{r+1}|=O_r(n^{r+1-\epsilon_{r,S}})$$ for some non-explicit $\epsilon_{r,S}$. This explicit power-savings in the exponent (in a slightly more general weighted setting) implies our fourth theorem, which for $n$-step random walks $X=g_1\cdots g_n$ (for independent $G$-valued random variables $g_1,\ldots,g_n$) converts polynomial upper bounds $\sup_g\mathbb{P}(X=g)\le n^{-C}$ and inverse theorems describing the structure of the distributions of the steps $g_1,\ldots,g_n$ of $X$ when $\sup_g\mathbb{P}(X=g)\ge n^{-C'}$ into analogous theorems with $\sup_g\mathbb{P}(X=g)$ replaced with $\mathbb{P}(X\in S)$.

\begin{notat}\label{notat}
For an $n$-step random walk $X=g_1\cdots g_n\in G$, where the $g_i$ are drawn independently from  distributions $\mu_i$, we will notate
\begin{align*}
p_0&:=\inf_{i,g}(\mu_i(g):g\in\operatorname{supp}(\mu_i)).\\
    \rho&:=\sup_g\mathbb{P}(X=g),\text{ and }\\
    \rho_S&:=\mathbb{P}(X\in S).
\end{align*}
\end{notat}
\begin{thmx}
\label{thm:selfirreduciblecomb}
Let $G$ be a group, and $g_1,\ldots,g_n\in G$ be independent elements chosen according to some distributions $\mu_1,\ldots,\mu_n$. Let $X=g_1\cdots g_n$. Then if $S\subset G$ is a subset of self-translate dimension $\le k$ of complexity $C$, then
$$\rho_S \le O_{C,k}(p_0^{-1}\rho^{1/(2^{k+1}-1)})$$
provided $p_0\ne 0$. Finally, if $X_0,\ldots,X_k$ are independent $G$-valued random variables and $\rho_i:=\sup_g\mathbb{P}(X_i=g)$, then
$$\mathbb{P}(\prod_{i=0}^k X_i\in S)=O_{C,k}(\max_i \rho_i^{1/2^i}).$$
\end{thmx}
This theorem is purely combinatorial, and will follow essentially identically to the proof of the analogous statement \cite[Theorem 5.1]{fox2021geometric} (with a minor optimization included, see \Cref{rem:betterexp}).

To apply this result, we need to know an upper bound on $\dimself(S)$. As any partition of a set containing unboundedly large arithmetic progressions contains unboundedly large arithmetic progressions (by van der Waerden's theorem on arithmetic progressions), it is easy to see that just like for $\dimrob(S)$, we have $\dimself(S)=\infty$ if $S$ contains unboundedly large arithmetic progressions. Assuming this doesn't happen, our final theorem compares these two notions of dimension with $\dim(S)$ in the o-minimal setting.
\begin{thmx}
\label{thm:ominimaldimension}
Suppose that $S\subset G$ is a definable subset of a definable group in an o-minimal structure $\mathbb{M}_{\mathcal{F}}$. Then if $S$ does not contain unboundedly large arithmetic progressions, we have
$$\dimrob(S)\le \dimself(S)\le \dim(S).$$
In a definable family of subsets of $G$, there is a constant $C$ such that every subset $S$ not containing unboundedly large arithmetic progressions has self-translate dimension $\le \dim(S)$ of complexity $C$.
\end{thmx}

Our strategy for proving $\dimself(S)\le \dim(S)$ is broadly the same as in \cite{fox2021geometric}, essentially showing $S$ satisfies the analogue of ``self-translate dimension $\le k$ of complexity $C$'' with self-translate dimension replaced everywhere by dimension, and the complexity $C$ being bounded in definable families.

We postpone applications to random walks in definable groups to \Cref{sec:probapp}, and leave the remainder of the article to proving the main theorems above.

\textbf{Acknowledgements:} I would like to thank Artem Chernikov for helpful comments, Matthew Kwan for helpful remarks on the exposition, Isaac Goldbring for pointing out an example of a definable group not satisfying (NSS), and Thomas Scanlon for additional references.

\section{Applications to random walks in definable groups}
\label{sec:probapp}
For a group $G$ and a family of probability distributions on $G$, we can create an $n$-step random walk $X=g_1\cdots g_n$ by taking each $g_i$ independently from some distribution $\mu_i$ inside the family.  Littlewood-Offord theory is the study of anti-concentration properties for $X$. Formally, for a subset $S\subset G$, 
the main questions are the following.
\begin{itemize}
    \item (Forward Question) What is the best upper bound on $\mathbb{P}(X\in S)$ independent of $\mu_1,\ldots,\mu_n$?
    \item (Inverse Question) If $\mathbb{P}(X\in S)$ is large, then what can we say about $\mu_1,\ldots,\mu_n$?
\end{itemize}

Classically, the forward question was considered first by Littlewood and Offord \cite{LO43} for $G=(\mathbb{R}^d,+)$, $S=B(z,r)$ a ball of radius $r$ centered around $z$, and the family of distributions $\{\mu_{v}:v\in \mathbb{R}^d, \|v\|\ge 1\}$, where $\mu_v(v)=\mu_v(-v)=\frac{1}{2}$. Then $X=\sum_{i=1}^n \epsilon_iv_i$ is a random signed sum with $(\epsilon_1,\ldots,\epsilon_n)\in \{-1,1\}^n$ a uniformly random sequence of signs, and $v_i\in \mathbb{R}^d$ fixed vectors with $\|v_i\|\ge 1$ for all $i$. The optimal forward theorem was obtained by Erd\H{o}s \cite{Erd45} for $d=1$ and by Kleitman \cite{K70} for general $d$ (with the asymptotically best multiplicative constant later obtained by Frankl-F\"uredi \cite{FF88}), showing $$\mathbb{P} (\|\sum_{i=1}^n\epsilon_i v_i -z\|\le r)= O_{d,r}(n^{-1/2}).$$
Up to the multiplicative constant, this is also the best possible upper bound for $\mathbb{P}(\sum_{i=1}^n \epsilon_iv_i=z)$, so the result can be appropriately interpreted as saying that such random walks cannot concentrate inside balls much better than they can concentrate on individual points.


Inverse results in this setting were first considered in the estimation of singularity probabilities for random matrix ensembles by Tao and Vu \cite{TV09a,TV09b}. The sharpest inverse result, established by Nguyen and Vu \cite{NV11}, says that for $\epsilon<1$ and $C>0$, if $v_1,\ldots,v_n$ are non-zero elements of a torsion-free abelian group $G$ (such as $(\mathbb{R}^d,+)$), then if $\rho:=\sup_z\mathbb{P}(\sum_{i=1}^n \epsilon_iv_i=z) \ge n^{-C}$, there exists a symmetric proper generalized arithmetic progression $Q$ of rank $r=O_{C,\epsilon}(1)$ containing all but $\epsilon n$ of the elements of $V$ (counted with multiplicity) such that $|Q|= O_{C,\epsilon}(\rho^{-1}n^{-r/2})$.

In \cite{fox2021geometric}, Fox, Kwan and the present author studied for subsets $S\subset \mathbb{R}^d$ anti-concentration results for $\mathbb{P}(\sum_{i=1}^n \epsilon_iv_i\in S)$ assuming only $v_i\ne 0$ for all $i$. The possibility of arbitrarily small steps $v_i$ means that the additive structure of $S$ plays a much greater role.

\begin{exmp}
\label{exmp:intervalbad}
For $S=B(z,r)$ there is no non-trivial upper bound on $\mathbb{P}(X\in S)$ because taking $v_1=z$ and $\|v_i\|\le r/n$ for $i=2,\ldots,n$ gives $\mathbb{P}(X\in S)\ge \frac{1}{2}$. In fact, we can't get any non-trivial upper bound for $\mathbb{P}(X\in S)$ if $S\subset \mathbb{R}^d$ contains a line segment, since we can take $v_1$ to be the midpoint of the segment and $v_2,\ldots,v_n$ to be small in the direction of the segment. Similar issues can arise if $S$ is allowed to contain oscillatory sets such as $y=\sin(x)$, or if $S$ has a fractal structure.
\end{exmp}

In \cite{fox2021geometric} it was shown that definability in an o-minimal structure $\mathbb{R}_{\mathcal{F}}$ extending $\mathbb{R}_{alg}:=(\mathbb{R},<,\times,+)$ (such as $\mathbb{R}_{alg}$ or $\mathbb{R}_{an,\exp}$), together with the necessary condition of not containing any line segment, is sufficient to obtain non-trivial forward and inverse theorems for $\mathbb{P}(X\in S)$. 
In the following, we note that all reasonable notions of dimension agree for definable sets in $\mathbb{R}_{\mathcal{F}}$ (see \Cref{sec:dim}), and $\rho,\rho_S$ were defined previously in \Cref{notat}.
\begin{thm}[Fox, Kwan, S. \cite{fox2021geometric}]\label{thm:FKS} Let $v_1,\ldots,v_n\in \mathbb{R}^d$ be non-zero vectors, and let $S\subset \mathbb{R}^d$ be a dimension $k$ set definable in an o-minimal structure $\mathbb{R}_{\mathcal{F}}$ extending $\mathbb{R}_{alg}$ such that $S$ does not contain any line segment. Let $(\epsilon_1,\ldots,\epsilon_n)\in \{-1,1\}^n$ be a uniformly random sequence of signs, and $X=\sum_{i=1}^n \epsilon_iv_i$. Then
\begin{enumerate}
    \item (Forward theorem) $\rho_S=n^{-1/2+o_S(1)}$.
    \item (Inverse bootstrapping theorem\footnote{The exponents on $\rho$ and $\rho_i$ here are slightly optimized from \cite{fox2021geometric} with an essentially identical proof, see \Cref{rem:betterexp}.}) $\rho_S = O_S(\rho^{1/(2^{k+1}-1)})$.
    Additionally, if $X_0,\ldots,X_k$ are independent $\mathbb{R}^d$-valued random variables and $\rho_i=\sup_g\mathbb{P}(X_i=g)$, then $$\mathbb{P}(X_0+\cdots+X_k\in S)= O_S(\max_i \rho_i^{1/2^i}).$$
\end{enumerate}
If $S$ is part of a definable family of subsets of $G$, then the $o_s$ and $O_S$ constants can be taken to be uniform in the family.
\end{thm}
As explained in the introduction, a concrete example of the last part of the theorem is that the ellipses in the definable family $\{(x,y):ax^2+by^2=1\}_{a,b>0}$ yield the same $o_S$ and $O_S$ constants.

The inverse bootstrapping theorem allows an existing inverse theorem (such as the one of Nguyen-Vu mentioned earlier) for the structure of $v_1,\ldots,v_n$ when $\rho$ is large to be bootstrapped to a structural theorem for when $\rho_S$ is large (since $\rho_S$ is bounded above by a function of $\rho$). In the same way, it can also bootstrap a forward result upper bounding $\rho$, such as the bound $\rho=O(n^{-1/2})$, to a forward result upper bounding $\rho_S$, although the result will not be as strong as the forward theorem $\rho=n^{-1/2+o_S(1)}$.

The proof of the forward theorem $\rho=n^{-1/2+o_S(1)}$ from \cite{fox2021geometric} requires the inverse bootstrapping theorem, together with a deep theorem of Pila \cite{P09} concerning the interaction of sets definable in an o-minimal structure with bounded rank generalized arithmetic progressions, and progress by Kane \cite{Kan14} on the Gotsman-Linial conjecture. Unfortunately the latter two are quite particular to the case $G=(\mathbb{R}^d,+)$, so we do not know at the moment how to generalize this stronger forward theorem to other definable groups. 

\Cref{thm:trajectory} and \Cref{thm:ominimaldimension}, in conjunction with \Cref{thm:selfirreduciblecomb} directly generalize the inverse bootstrapping theorem from $G=(\mathbb{R}^d,+)$ to arbitrary definable groups.\footnote{Technically, the scope of \Cref{thm:trajectory} is only for o-minimal expansions of $\mathbb{R}_{an,\exp}$, which does not contain all o-minimal expansions of $\mathbb{R}_{alg}$ as covered by \Cref{thm:FKS}. \Cref{thm:ominimaldimension}, which does include o-minimal expansions of $\mathbb{R}_{alg}$, has the condition of avoiding a line segment in $(\mathbb{R}^d,+)$ replaced with the stronger condition of not containing unboundedly large arithmetic progressions, but fortunately, these latter conditions are in fact equivalent in $(\mathbb{R}^d,+)$ by the Uniform Finiteness Principle \Cref{thm:unifbound} applied to the definable family of subsets $\{S\cap \{y+tv:t\in \mathbb{R}\}\}_{y\in \mathbb{R}^d,v\in \mathbb{R}^d\setminus 0}$ of $\mathbb{R}^d$.}
Hence we can take existing forward and inverse theorems in arbitrary groups for $\sup_g \mathbb{P}(X=g)$, and bootstrap them to forward and inverse results for $\mathbb{P}(X\in S)$ for $S\subset G$ a definable subset of a definable group in an o-minimal structure.

\subsection{Forward theorems for random walks in definable groups}
Ju\v{s}kevi\v{c}ius and \v{S}emetulskis \cite{OptimalGroup} considered for an arbitrary group $G$ a random product $X=\prod_{i=1}^n g_i^{\pm 1}$, where $g_i\in G$, and each exponent taken independently to be $1$ or $-1$ with probability $1/2$, obtaining the following result.
\begin{thm}[Ju\v{s}kevi\v{c}ius, \v{S}emetulskis {\cite[Corollary 1]{OptimalGroup}}]
Under this model, if every element $g_i$ has order at least $s$, then we have $$\rho \le 3\max(s^{-1},n^{-1/2}).$$
\end{thm}
We note that the same inequality in $GL_d(\mathbb{R})$ was obtained earlier in Tiep and Vu \cite{TiepVu} with the constant $141$ instead of $3$.
\begin{cor}
\label{cor:forward1}
Under the same assumptions, let $S\subset G$ be a definable subset of a definable group avoiding unboundedly large arithmetic progressions in an o-minimal structure $\mathbb{M}_{\mathcal{F}}$ , and let $k=\dim(S)$. Then we have
$$\rho_S = O_S(\max(s^{-1/2^k},n^{-1/2^{k+1}})).$$
If $S$ is part of a definable family of subsets of $G$, then the constant $O_S$ can taken to be uniform for elements of the family.
\end{cor}
\begin{proof}
Apply \Cref{thm:ominimaldimension} and the second part of \Cref{thm:selfirreduciblecomb} to $X_i=\prod_{j\in I_i} g_j$ for an equipartition $I_0\sqcup \ldots \sqcup I_k=\{1,\ldots,n\}$, noting that $\rho_i \le 3\max(s^{-1},\lceil n/(k+1)\rceil^{-1/2})$ for each $i$ by the above theorem.
\end{proof}
We conjecture that this bound is not tight, and make a number of concrete conjectures of increasing strength to this effect.
\begin{conjecture}
Under the same assumptions,
\begin{enumerate}[a)]
    \item There exists $\epsilon>0$ depending on $S$ such that $\rho_S= O_S(\max(s^{-1},n^{-1/2}))^{(1/2^k)+\epsilon}$
    \item There exists  $\epsilon>0$ depending on $k,G$ such that $\rho_S= O_S(\max(s^{-1},n^{-1/2}))^{(1/2^k)+\epsilon}$
    \item $\rho_S= \max(s^{-1},n^{-1/2})^{1-o_S(1)}$
    \item $\rho_S= O_S(\max(s^{-1},n^{-1/2}))$.
\end{enumerate}
\end{conjecture}
As was mentioned earlier, for $G=(\mathbb{R}^d,+)$ the analogue of c) was established in \cite{fox2021geometric} (see \Cref{thm:FKS} above). Also the analogue of d) for $G=(\mathbb{R}^d,+)$ and $k=1$ readily follows from \cite[Theorem 1.9]{fox2021geometric}, the analogous result for $S$ a strictly convex plane curve, by decomposing a generic projection of $S$ to $\mathbb{R}^2$ into convex arcs (which is possible by standard results on definable cell decompositions, see \cite{vdD98}).

Nguyen \cite{nguyen2017anticoncentration} considered a more general product $X=\prod_{i=1}^ng_i$ in an arbitrary group $G$, where $g_i$ are selected from arbitrary distributions $\mu_i$, and obtained the following result.
\begin{thm}[Nguyen, {\cite[Theorem 1.5]{nguyen2017anticoncentration}}]
For any $\delta>0$ there exist $n_0$ and $0<\epsilon<1$ such that the following holds for $n\ge n_0$. Assume that $p_0\ge n^{-\epsilon}$ and each $\operatorname{supp}(\mu_i)$ contains a pair of elements $a_i,a_i'$ such that $a_ia_i'^{-1}$ has order at least $s$. Then
$$\rho = O_{\delta}(\max(s^{-1},n^{-1/2+\delta})).$$
\end{thm}
\begin{cor}
\label{cor:forward2}
Under the same assumptions, let $S\subset G$ be a definable subset of a definable group avoiding unboundedly large arithmetic progressions in an o-minimal structure $\mathbb{M}_{\mathcal{F}}$ , and let $k=\dim(S)$. Then we have
$$\rho_S = O_{S,\delta}(\max(s^{-1/2^k},n^{-1/2^{k+1}+\delta})).$$
If $S$ is part of a definable family of subsets of $G$, then the constant $O_S$ can taken to be uniform for elements of the family.
\end{cor}
\begin{proof}
Apply \Cref{thm:ominimaldimension} and the second part of \Cref{thm:selfirreduciblecomb} to $X_i=\prod_{j\in I_i} g_j$ for an equipartition $I_0\sqcup \ldots \sqcup I_k=\{1,\ldots,n\}$, noting that the above theorem with $2^k\delta$ in place of $\delta$ gives $\rho_i \le O_{\delta}(\max(s^{-1/2},n^{-1/2+2^k\delta})$ for each $i$.
\end{proof}
\subsection{Inverse theorem for random walks in definable groups}
For general groups $G$, there is a robust theory of ``approximate groups'' developed by Breuillard, Green, and Tao \cite{ApproximateGroup} of subsets $N\subset G$ such that $N+N$ can be covered by a bounded number of cosets of $N$. The key notion is a ``coset nilprogression in $C$-normal form'', whose precise definition we will not repeat here.
Nguyen \cite{nguyen2017anticoncentration} proved the following result for random walks in arbitrary groups $G$ (we also refer to \cite{nguyen2017anticoncentration} for the definition of a coset nilprogression in $C$-normal form).
\begin{thm}[Nguyen, {\cite[Theorem 1.10]{nguyen2017anticoncentration}}]
Let $A>0$ and $0<\epsilon<1$ be constants, and assume that $p_0 \ge n^{-\epsilon^3}$ and
$$\rho \ge n^{-A}.$$
Then there is a coset nilprogression $HP$ with the following properties.
\begin{enumerate}
    \item $P$ is in $C$-normal form with $C=O_{A,\epsilon}(1)$ with rank and step $r,s=O_{A,\epsilon}(1)$.
    \item $|HP|=O_{A,\epsilon}(\rho^{-1})$
    \item There is a finite set $Y$ of cardinality $|Y|=O_{A,\epsilon}(1)$ and consecutive indices $i_0,\ldots,i_0+n'$ with $n'=n^{1-O_A(\epsilon)}$ such that the following holds: for each $a,a'\in \operatorname{supp}(\mu_i)$, $i_0\le i \le i_0+n'$ there exists a permutation $\sigma$ of $Y$ such that for all $x\in Y$ we have
    $$aa'^{-1}\in xHP(\sigma(x))^{-1}.$$
\end{enumerate}
\end{thm}
\begin{cor}
\label{cor:inverse}
Let $A>0$ and $0<\epsilon<1$ be constants, and assume that $p_0 \ge n^{-\epsilon^3}$. Let $S\subset G$ be a definable subset of a definable group avoiding unboundedly large arithmetic progressions in an o-minimal structure $\mathbb{M}_{\mathcal{F}}$, and let $k=\dim(S)$. Suppose that
$$\rho_S \ge n^{-A}.$$ Then there is a constant $n_S$ depending only on $S$, such that for $n\ge n_S$, there is a coset nilprogresison $HP$ with the following properties.
\begin{enumerate}
    \item $P$ is in $C$-normal form with $C=O_{A,\epsilon,k}(1)$ with rank and step $r,s=O_{A,\epsilon,k}(1)$.
    \item $|HP|=O_{A,\epsilon,S}(p_0^{-1}\rho_S^{-(2^{k+1}-1)})$
    \item There is a finite set $Y$ of cardinality $|Y|=O_{A,\epsilon,k}(1)$ and consecutive indices $i_0,\ldots,i_0+n'$ with $n'=n^{1-O_{A,k}(\epsilon)}$ such that the following holds: for each $a,a'\in \operatorname{supp}(\mu_i)$, $i_0\le i \le i_0+n'$ there exists a permutation $\sigma$ of $Y$ such that for all $x\in Y$ we have
    $$aa'^{-1}\in xHP(\sigma(x))^{-1}.$$
\end{enumerate}
If $S$ is part of a definable family of subsets of $G$, then the dependencies of the constants on $S$ can taken to be uniform for elements of the family.
\end{cor}
\begin{proof}
Applying \Cref{thm:ominimaldimension} and the first part of \Cref{thm:selfirreduciblecomb}, we get $\rho\ge (\rho_Sn^{-1})^{(2^{k+1}-1)}C_S^{-1}$, so we can apply the above theorem with  $(2^{k+1}-1)(A+1)+1$ in place of $A$ (the last $+1$ is used to absorb $C_S$ in the condition $n \ge n_S$).
\end{proof}

\section{Proof of \Cref{thm:selfirreduciblecomb}}
In this section, we prove \Cref{thm:selfirreduciblecomb}. We first begin with some preliminary results.
\label{sec:ProofFKS}

\begin{lem}
\label{rem:equivdef}
Suppose that $S$ has self-translate dimension $\le k$ of complexity $C$. Then $g_1Sg_2$ also has self-translate dimension $\le k$ of complexity $C$ for all $g_1,g_2\in G$. In particular, the condition $S_i\cap gS_i$ having self-translate dimension $\le k-1$ of complexity $C$ for all $\operatorname{id}_G\ne g \in G$ in \Cref{defn:selftranslate} implies the seemingly more general condition that $gS_i \cap g'S_i$ has self-translate dimension $\le k-1$ of complexity $C$ for all $g\ne g'$. 
\end{lem}
\begin{proof}
Indeed, we can write $g_1Sg_2=(g_1Sg_1^{-1})g_1g_2$, which reduces to showing that this notion is stable under conjugation and right multiplication. The former is clear since conjugation is an automorphism of $G$, and the latter follows because right multiplication commutes with the left-multiplication of elements of $G$ in \Cref{defn:selftranslate}.
\end{proof}

\begin{lem}[Fox, Kwan, S. {\cite[Theorem 3.1]{fox2021geometric}}]\label{lem:decouple}
Given an event $E(Y,Z)$ which depends on the outcomes of independent random variables $Y,Z$, setting $\sup_y\mathbb{P}(Y=y) =\mu$ and $\sup_{y\ne y'}\mathbb{P}(E(y,Z)\wedge E(y',Z))=\lambda$, we have
$$\mathbb{P}(E(Y,Z))\le \lambda^{1/2}+2\mu.$$
\end{lem}
\begin{lem}\label{lem:partite}
For $0<\lambda_0,\ldots,\lambda_k\le p_0$ such that $\rho\le \prod_{i=0}^k \lambda_i$, there exists a partition $I_0\sqcup \ldots \sqcup I_k$ of $\{1,\ldots,n\}$ into contiguous intervals so that setting $X_i=\prod_{j\in I_i}g_j$, we have $\rho_i:=\sup_g\mathbb{P} (X_i=g)\le p_0^{-1}\lambda_i$ for $0\le i \le k$. In particular, if $\rho^{1/(2^{k+1}-1)}p_0^{-1} \le 1$ then we can take $\rho_i \le \rho^{2^i/(2^{k+1}-1)}p_0^{-1}$.
\end{lem}
\begin{proof}
We induct on $k$. The result is trivial for $k=0$, so assume $k>0$ and the result is true for $k-1$. First, note that
$\sup_g\mathbb{P}(\prod_{i=1}^1 g_i=g)\ge p_0 \ge \lambda_0 \ge \rho=\sup_g\mathbb{P}(\prod_{i=1}^n g_i=g)$. Hence there is an index $1 \le \ell\le n-1$ such that $\sup_g\mathbb{P} (\prod_{i=1}^\ell g_i=g) \ge \lambda_0 \ge \sup_g\mathbb{P} (\prod_{i=1}^{\ell+1} g_i=g)$.

Then because $\sup_g\mathbb{P} (\prod_{i=1}^{\ell+1} g_i=g)\ge p_0\sup_g\mathbb{P} (\prod_{i=1}^\ell g_i=g)$, we must have $$\lambda_0\le \sup_g\mathbb{P} (\prod_{i=1}^\ell g_i=g) \le p_0^{-1}\lambda_0.$$
Now $\rho \ge \sup_g \mathbb{P}(\prod_{i=1}^\ell g_i=g)\sup_g \mathbb{P}(\prod_{i=\ell+1}^ng_i=g)\ge \lambda_0\sup_g\mathbb{P}(\prod_{i=\ell+1}^n g_i=g)$, so we conclude that $\sup_g \mathbb{P}(\prod_{i=\ell+1}^n g_i=g) \le \rho/\lambda_0$ and the result follows by taking $I_0=\{1,\ldots,\ell\}$ and applying the inductive hypothesis to $\{\ell+1,\ldots,n\}$.
\end{proof}
We now prove \Cref{thm:selfirreduciblecomb}, identically to the analogous result for $(\mathbb{R}^d,+)$ proved in \cite[Theorem 1.14]{fox2021geometric}.
\begin{proof}[Proof of \Cref{thm:selfirreduciblecomb}]
If $p_0^{-1}\rho^{1/(2^{k+1}-1)}>1$ then the first part of the theorem is trivial. Otherwise, we may apply the last part of \Cref{lem:partite} to obtain a partition $I_0\sqcup \ldots \sqcup I_k$ of $\{1,\ldots,n\}$ into contiguous intervals so that setting $X_i=\prod_{j\in I_i}g_j$, we have $\rho_i:=\sup_g\mathbb{P}(X_i=g)\le  \rho^{2^i/(2^{k+1}-1)}p_0^{-1}$. Then because $X=\prod_{i=0}^kX_i$, the first part of the theorem follows from the second part.

To prove the second part, we induct on $k$. For $k=0$ the result is trivial by the union bound, so assume now that $k\ge 1$. By the definition of self-translate dimension $\le k$ of complexity $C$ and the union bound, we may assume that $S$ itself has the property that $S\cap gS$ is of self-translate dimension $\le k-1$ of complexity $C$ for all $\text{id}_G\ne g\in G$. Write $Y=X_0$ and $Z=\prod_{i=1}^k X_i$, and consider the event $E(Y,Z)=\mathbb{P}(YZ\in S)$. Then by \Cref{lem:decouple}, we have
\begin{align*}\mathbb{P}(\prod_{i=0}^k X_i\in S)=\mathbb{P}(E(Y,Z))&\le \sup_{y \ne y'}\mathbb{P}(E(y,Z)\wedge E(y',Z))^{1/2}+2\sup_y(Y=y)\\
&\le\sup_{y \ne y'}\mathbb{P}(Z\in y^{-1}S\cap y'^{-1}S)^{1/2}+2\rho_0\end{align*}
Now, by \Cref{rem:equivdef} we have $y^{-1}S\cap y'^{-1}S$ has self-translate dimension $\le k-1$ with complexity $C$. By the inductive hypothesis we have $$\sup_{y \ne y'}\mathbb{P}(Z\in y^{-1}S\cap y'^{-1}S)^{1/2}= O_{C,k}(\max_{1\le i \le k}\rho_i^{1/2^{i-1}})^{1/2}=O_{C,k}(\max_{1\le i \le k}\rho_i^{1/2^i}).$$
Hence $\mathbb{P}(\prod_{i=0}^k X_i\in S) \le O_{C,k}(\max_i \rho_i^{1/2^i})$ as desired.
\end{proof}
\begin{rem}
\label{rem:betterexp}
In \cite{fox2021geometric}, the inverse result was stated as $O_{C,k}(\max_i \rho_i^{1/2^k})$ even though the same proof yields the slightly better $O_{C,k}(\max_i \rho_i^{1/2^i})$. Because of this, the $\rho_i$ were chosen to be equal rather than separated by powers of $2$. This minor optimization yields the slightly better exponents for the inverse theorems in \cite{fox2021geometric}, which we've included into the recollection of the result from \cite{fox2021geometric} (\Cref{thm:FKS}).
\end{rem}

\section{o-minimality definitions}
\label{sec:omindefs}

A structure $\mathbb{M}_{\mathcal{F}}=(M,<,\mathcal{F})$ extending a linear order $(M,<)$ is the additional data of a collection $\mathcal{F}$ of functions  $f:M^{n_f}\to M$ of various arities $n_f$. The structure $\mathbb{M}_{\mathcal{F}}$ identifies certain subsets $S_\phi\subset M^d$ as ``definable'' when they are the locus where a first-order formula $\phi(x_1,\ldots,x_d)$ in the structure is true. Geometrically, these are finite Boolean combinations of subsets $\{f_1>g_1,\ldots,f_a>g_a,f_{a+1}=g_{a+1},\ldots,f_{a+b}=g_{a+b}\}\subset M^d$ where $f_i,g_i$ are arbitrary compositions of constants from $M$ and functions from $\mathcal{F}$ (which correspond to quantifier-free $\phi$), together with alternating projections and complementations of these sets (which correspond to arbitrary $\phi$ written in prenex normal form).

\begin{defn}
A structure $\mathbb{M}_{\mathcal{F}}$ extending a dense linear order $(M,<)$ is said to be o-minimal if every definable subset $S\subset M$ is a finite union of points and intervals with endpoints in $M\cup \{-\infty,\infty\}$ (open on an endpoint at $-\infty$ or $\infty$, but allowed to be open or closed on an endpoint in $M$).
\end{defn}
Although this may seem like a rather weak hypothesis, it guarantees that the definable subsets are geometrically tame, in particular avoiding certain oscillatory and fractal behaviour. For example, definable sets in $\mathbb{R}_{\mathcal{F}}$ are always a finite union of topological manifolds. 

\begin{exmp}
\label{exmp:important}
Some important examples of o-minimal structures $\mathbb{R}_{\mathcal{F}}$ are the following.
\begin{itemize}
    \item The semi-algebraic structure $\mathbb{R}_{alg}:=\mathbb{R}_{\{+,\times\}}$. By the Tarski-Seidenberg theorem \cite{T51}, the definable subsets of $\mathbb{R}^d$ are the semi-algebraic subsets, which are finite Boolean combinations of subsets $\{f_1>0,\ldots,f_a>0,f_{a+1}=0,\ldots,f_{a+b}=0\}$ for polynomials $f_1,\ldots,f_{a+b}$.
    \item The restricted analytic-exponential structure $\mathbb{R}_{an,\exp}:=\mathbb{R}_{\mathcal{F}}$ of van den Dries and Miller \cite{DM94}, for $$\mathcal{F}=\{\exp,+,\times\}\cup \{f|_{[0,1]^r}: f:N\to \mathbb{R} \text{ analytic, where $N\subset \mathbb{R}^r$ is an open neighborhood of $[0,1]^r$}\}.$$
\end{itemize}
\end{exmp}

A definable group $G$ in an o-minimal structure $\mathbb{M}_{\mathcal{F}}$ is a definable subset $G\subset M^d$ for some $d$, equipped with a group structure such that the graph of the multiplication map in $M^{3d}$ is definable. 
  For example when $\mathbb{R}_{\mathcal{F}}=\mathbb{R}_{alg}$, $G$ is a semi-algebraic group such as $GL_d(\mathbb{R})$, $GL_d(\mathbb{C})$, the unitary group $U_d(\mathbb{C})$, $(\mathbb{R}^d,+)$, or an abelian variety. Since the study of definable groups was first initiated by Pillay \cite{Pillay}, their properties and classification have become a cornerstone of model theory. We refer the reader to the surveys of Otero \cite{otero_2008} and, for o-minimal structures extending a real-closed field, Conversano \cite{conversano2020groups}.

To decrease notational overhead, rather than working with ``parametrized formulas'' and ``definable families'' as in \cite{Cos00} and \cite{vdD98}, we choose instead to work with the equivalent notion of ``complexity'', which is the o-minimal analogue of degree for algebraic varieties.
\begin{defn}
We say that $\mathbb{M}_\mathcal{F}$ is finitely generated if $|\mathcal{F}|<\infty$. In this case, the complexity of a formula $\phi$, denoted $\operatorname{comp}_{\mathcal{F}}(\phi)$, is the number of symbols used in the defining formula of $\phi$ (where constants count as one symbol). For a definable set $S\subset \mathbb{R}^d$, we define $\operatorname{comp}_{\mathcal{F}}(S)=\min_{\phi: S=S_\phi}\operatorname{comp}_{\mathcal{F}}(\phi)$.
\end{defn}
 We note that given a formula of fixed combinatorial type, varying the parameters in $M$ yields a definable family, and a definable family is defined by a formula so the fibers are of bounded complexity. Hence these notions are interchangeable. Also, if $S=S_\phi$ is defined in an o-minimal structure $M_{\mathcal{F}}$, then it is also definable in the finitely generated o-minimal structure $M_{\mathcal{F}'}$, where $\mathcal{F}'$ contains the functions that appear in $\phi$, so working with finitely generated o-minimal structures is never a serious restriction.

\section{Facts about o-minimality}
\label{sec:dim}
Fix a dense linear order $(M,<)$, and let $\mathbb{M}_\mathcal{F}$ always refer to a finitely generated o-minimal structure. Definable sets in $\mathbb{M}_{\mathcal{F}}$ are very well-behaved, and we collect the properties of such sets here.
\subsection{Uniform Finiteness}
The notion of complexity is connected to o-minimality primarily through the following fact, which bounds the size of a finite definable set in terms of the complexity.
\begin{fact}[Uniform Finiteness Principle {\cite[Ch.3 Lemma 2.13]{vdD98}}]\label{thm:unifbound}If $S\subset M^d$ is definable in $\mathbb{M}_{\mathcal{F}}$, and $|S|<\infty$, then $|S|=O_{d,\mathcal{F},\operatorname{comp}_{\mathcal{F}}(S)}(1)$.
\end{fact}
\subsection{Definable functions}
Definable functions are more general than just compositions of functions from $\mathcal{F}$.
\begin{defn}
Given subsets $A\subset M^a$ and $B\subset M^b$ and a function $f:A\to B$, we say $f$ is definable if $A,B$ are definable and the graph $$\Gamma(f):=\{(x,f(x)):x\in A\}\subset M^a\times M^b$$
 is definable.
\end{defn}
\begin{exmp}
In $\mathbb{R}_{alg}$, the function $y=x^{3/2}$ from $\mathbb{R}_{\ge 0}\to \mathbb{R}$ is definable because its graph is defined by the first order formula $$\phi(x,y):\exists t\in \mathbb{R}, t\ge 0 \wedge t^2=x \wedge t^3=y.$$
\end{exmp}
It is straightforward to check with this definition that
\begin{itemize}
    \item Compositions of definable functions are definable
    \item Restrictions of definable functions to definable sets are definable
    \item Images and pre-images of definable sets under definable functions are definable.
\end{itemize}
We will additionally need the following facts, showing that a definable curve is piecewise-continuous.
\begin{fact}[Monotonicity Theorem {\cite[Ch. 3 Theorem 1.2]{vdD98}}]
\label{fact:contshrink}
For an interval $I\subset M$ and a definable map $f:I\to M$, there is a finite set $F\subset I$ such that on each subinterval $J$ of $I\setminus F$ that $F$ splits $I$ into, $f|_J$ is either constant, or strictly monotone and continuous.
\end{fact}
\subsection{Definable groups}
Recall that a group $G\subset M^d$ is said to be definable in $\mathbb{M}_{\mathcal{F}}$ if the graph of its multiplication map in $M^{3d}$ is definable. 
\begin{rem}
\label{rem:Pillay}
By Pillay's structure theory for definable groups \cite{Pillay}, there is a unique topology on $G$ called the ``t-topology'' (possibly different than the induced topology of $\mathbb{R}^n$) making $G$ a topological group, such that there is an open set $U\subset G$ with $\dim(G\setminus U)<\dim(G)$, and the open sets in $U$ agree in both topologies. Furthermore, the t-topology is induced by finitely many definable open subsets $U_1,\ldots,U_\ell\subset M^k$ and charts $\phi_i:U_i\to G$ which are definable such that the transition maps are definable and continuous (here a map is definable if its graph is a definable set).

In particular if $\mathbb{M}_\mathcal{F}=\mathbb{R}_\mathcal{F}$ then $G$ is naturally a locally Euclidean topological group, and hence by the resolution to Hilbert's fifth problem by Gleason \cite{Gleason} and Montgomery-Zippen \cite{MontZip}, a smooth Lie group.
\end{rem}

Subgroups $H\subset G$ of definable groups are closed \cite[Proposition 2.7]{Pil87}, quotients $G/H$ of definable groups $G$ by definable normal subgroups $H$ are themselves definable groups by Edmundo \cite[Theorem 7.2]{EDMUNDO2003103} (which guarantees a definable family of coset representatives for $G/H$ in $G$), and group homomorphisms $G_1\to G_2$ between definable groups are always continuous in the ``t-topology'' (since \Cref{fact:contshrink} establishes continuity in some open subset of $G_1$, and we can use the continuous left-multiplication in $G_1$ and $G_2$ to deduce global continuity).

The following special case of a fact of Edmundo shows definable lifts exist along maps whose domain is a subset of a definable group.
\begin{fact}[Definable lifts across maps whose domain lies in a definable group {\cite[Theorem 7.2]{EDMUNDO2003103}}]
\label{fact:definablelift}
Let $Y\subset G$ be a subset of a definable group, and $\phi:Y\to X$ a definable surjection onto a definable set. Then given a definable map $\psi:Z\to X$, there is a lift $\widetilde{\psi}:Z\to Y$ so that $\phi \circ \widetilde{\psi}=\psi$.
\end{fact}

\subsection{Dimension}Definable sets have a good theory of dimension. 
\begin{defn}[{\cite[Ch.4 Exercise 1.17(1)]{vdD98}}]
If $S\subset M^b$ is definable, then the dimension of $S$, written $\dim(S)$, is the maximum $k$ such that the image of $S$ under at least one of the $\binom{b}{k}$ coordinate projections $M^b\to M^k$ has nonempty interior.
\end{defn}

When $\mathbb{M}_{\mathcal{F}}=\mathbb{R}_{\mathcal{F}}$, the o-minimal notion of dimension is equivalent to all other reasonable notions of dimension (e.g. Hausdorff dimension).
 We collect some basic facts about dimension in arbitrary o-minimal $\mathbb{M}_{\mathcal{F}}$ which we will use freely, all of which can be found in \cite{vdD98}. In this list, all sets are definable.
\begin{itemize}
\item If $S_1\subset S_2$ then $\dim(S_1)\le \dim(S_2)$.
\item $\dim M^d=d$.
\item $\dim(S)=0$ if and only if $S$ is a finite union of points.
\item If there is a definable bijection $S_1\to S_2$, then $\dim(S_1)=\dim(S_2)$.
\item If $S=S_1\cup S_2$ then $\dim S=\max(\dim S_1,\dim S_2)$.
\item For a definable subset $S\subset M^d$, the topological closure $\overline{S}$ is definable and $\dim(\overline{S}\setminus S)<\dim(S)$.
\end{itemize}
This last fact about dimension has the  following corollary (a version of which was proved in \cite[Fact 7.12]{fox2021geometric}) which we will need later.
\begin{cor}
\label{prop:XYk}
If $X,Y\subset M^d$ are dimension $k$ sets definable in an o-minimal structure $\mathcal{F}$ such that $\dim(X\cap Y)=k$, then there exists $z\in X \cap Y$ and an open neighborhood $N$ of $z$ such that $N\cap X=N\cap Y$.
\end{cor}
\begin{proof}
Let $W=X\cup Y \setminus (X\cap Y)$. Then $\dim(\overline{W}\setminus W)\le \dim(W)-1\le k-1$, so $(X\cap Y)\setminus \overline{W}=(X\cap Y)\setminus (\overline{W}\setminus W)$ is nonempty. Hence taking $N=M^d\setminus \overline{W}$ and $z\in (X\cap Y)\setminus \overline{W}$, the result follows.
\end{proof}

We also have the following fact which says that the points in the image of a definable set under a linear projection whose pre-images have dimension $k$ is definable of bounded complexity.
\begin{fact}[{\cite[Ch.4 Proposition 1.5]{vdD98}}]\label{prop:definableloci}
Let $L:M^a\times M^b\to M^b$ be the coordinate projection onto the last $b$ coordinates. Then for a definable set $S\subset M^{a}\times M^b$ and any $k$, the subset $T=\{x\in L(S):\dim(L|_S^{-1}(x))=k\}\subset L(S)$ is definable of complexity $O_{a,b,\mathcal{F},\operatorname{comp}_{\mathcal{F}}(S)}(1)$, and $\dim(L|_S^{-1}(T))=\dim T + k$.
\end{fact}
\section{Proof of \Cref{thm:NST}}
Let $G$ be a definable group in an o-minimal structure $\mathbb{M}_{\mathcal{F}}$. We want to show that $G$ has the ``no small torsion'' property \eqref{tag:NST}.

\begin{lem}
\label{lem:extension}
Let $H\subset G$ be a definable normal subgroup of a definable group $G$. Then if $H,G/H$ have the ``no small torsion'' property \eqref{tag:NST}, then $G$ also has this property.
\end{lem}
\begin{proof}
We work with the ``t-topology'' on all groups throughout.
Let $\pi:G\to G/H$ be the quotient map. By the \eqref{tag:NST} property for $G/H$ there is an open neighborhood $\operatorname{id}_{G/H}\in N_m\subset G/H$ such that $ N_m-\{\operatorname{id}_{G/H}\}$ contains no torsion of order $\le m$, and by the \eqref{tag:NST} property for $H$ there is an open neighborhood $\operatorname{id}_G\in N_m'\subset G$ such that $N_M'\cap H-\{\operatorname{id}_G\}$  contains no torsion of order $\le m$. Then $U=N_m'\cap \pi^{-1}(N_m)$ has the property that $U-\{\operatorname{id}_G\}$ contains no torsion of order $\le m$, since for $x\in U$ either $\pi(x)\in N_M-\{\operatorname{id}_{G/H}\}$, or else $x\in N_m'\cap \pi^{-1}(\operatorname{id}_{G/H})=N_m' \cap H$.
\end{proof}

\begin{lem}
\label{lem:finsolvab}
If $G$ is either finite or abelian, then $G$ satisfies the ``no small torsion'' property \eqref{tag:NST}.
\end{lem}

\begin{proof}
If $G$ is finite then the result is trivial. If $G$ is abelian, then the $k$-torsion points $G[k]$ form a subgroup of bounded exponent, which is finite by Strzebonski \cite[Proposition 6.1]{Str94} (it is well known that the definable choice assumption needed \cite{Str94} can be overcome by applying \cite[Theorem 7.2]{EDMUNDO2003103}, see for example the discussion in Otero's survey on definable groups \cite{otero_2008} after \cite[Theorem 5.1]{otero_2008}), and hence we may take the open neighborhood around the identity to be $G\setminus (\bigcup_{k=2}^m G[k]-\{\operatorname{id}_G\})$.

\end{proof}

\begin{lem}
\label{lem:GLn}
If $G$ is a definable subgroup of a product $\prod_{i=1}^k GL_{n_i}(\mathcal{R}_i)$ of general linear groups over real closed fields, then $G$ has the ``no small torsion'' property \eqref{tag:NST}.
\end{lem}
\begin{proof}
It suffices to show the property for $\prod_{i=1}^k GL_{n_i}(\mathcal{R}_i)$, and this would follow if we can establish each $GL_{n_i}(\mathcal{R}_i)$ has the \eqref{tag:NST} property. 

Because this is a first-order property, this is equivalent to the no small torsion property for $GL_{n}(\mathbb{R})$. But if $0\in B\subset \mathfrak{gl}_n(\mathbb{R})$ is an open ball around $0$ on which the matrix exponential map $\exp_{GL_n(\mathbb{R})}:\mathfrak{gl}_n(\mathbb{R})\to GL_n(\mathbb{R})$ is a homeomorphism onto an open subset of $GL_n(\mathbb{R})$, then $\exp_{GL_n(\mathbb{R})}(\frac{1}{k}B)-\{\operatorname{id}\}$ has no torsion of order $k$ (since for $0\ne A\in B$ we have $(\exp_{GL_n(\mathbb{R})}(\frac{1}{k}A))^k=\exp_{GL_n(\mathbb{R})}(A)\ne \exp_{GL_n(\mathbb{R})}(0)=\operatorname{id}$).
\end{proof}
Now we can prove \Cref{thm:NST}.

\begin{proof}[Proof of \Cref{thm:NST}]
By Pillay \cite[Proposition 5.2]{Pil87} the quotient of a definable group by the definably connected component of the identity is finite, so applying \Cref{lem:finsolvab} and \Cref{lem:extension} we may assume that $G_0$ is definably connected.

Let $G_0=G$, and let $G_i=G_{i-1}/Z(G_{i-1})$, the quotient of $G_{i-1}$ by its center (which is a definable normal subgroup). Then we claim that $G_{k}$ is centerless for some $k\le \dim(G)+1$. Indeed, whenever $\dim Z(G_{i-1}) \ge 1$, we have $\dim G_i = \dim  G_{i-1}+\dim Z(G_{i-1})>\dim(G_{i-1})$ by the last part of \Cref{prop:definableloci}. Hence for some $i\le \dim G$ we must have $\dim Z(G_i)=0$, so by \Cref{thm:unifbound} $Z(G_i)$ is finite.

We claim that for $k=i+1$ we have $G_{k}$ is centerless. Indeed, if $a\in G_i$ descends to a central element in $G_{k}=G_i/Z(G_i)$, then the continuous map $b\mapsto aba^{-1}b^{-1}$ is a definable continuous map $G_i\to Z(G_i)$, which by the discreteness of $Z(G_i)$ and the definable connectedness of $G_i$ implies that the map is constant, and hence $a\in Z(G_i)$.

Repeatedly applying \Cref{lem:finsolvab} and \Cref{lem:extension}, we see that it suffices to prove the \eqref{tag:NST} property for $G_k$. By Peterzil, Pillay, Starchenko \cite[Theorem 3.1 and Theorem 3.2]{PPS00}, $G_k$ is a subgroup of a product of general linear groups over real closed fields, so we conclude by \Cref{lem:GLn}.
\end{proof}

\section{Locally avoiding unboundedly large arithmetic progressions}
\label{sec:expmatrixarcs}
In this section, we show the following proposition, which allows us to consider \Cref{thm:trajectory} and \Cref{thm:generalominimal} at the same time by working with a common weakening of the hypotheses on $S$.
\begin{prop}\label{prop:balls}
If $G$ is a definable group in  $\mathbb{R}_{an,\exp}$ and $S\subset G$ is a definable subset that contains no exponential arc, then for all $h\in S$ there exists an open neighborhood $N_h$ of $h$ in $G$ such that $N_h\cap S$ does not contain unboundedly large arithmetic progressions starting at $h$.
\end{prop}
First, we will need a lemma concerning the definability of the Lie group exponentiation map.
\begin{lem}
\label{lem:7.1}
If $G$ is a definable group in the o-minimal structure $\mathbb{R}_{an,\exp}$, then there is an open ball $0\in B\subset \mathfrak{g}$ around $0$ such that $\exp_{G}|_B$ is definable in  $\mathbb{R}_{an,\exp}$, and a homeomorphism onto an open subset of $G$.
\end{lem}
\begin{proof}
By van den Dries and Miller \cite[Corollary 6.12, (ii)]{DM94} and \cite[Theorem 8.8, (II)]{DM94}, every $\mathbb{R}_{an,\exp}$-definable function is analytic on some open set, so we can find a similar presentation to \Cref{rem:Pillay} where the transition maps, inverse map, and multiplication map of $G$ are analytic (see Pillay \cite[Remark 2.6]{Pillay}). By the Cauchy-Kowalewski theorem there exists an open neighborhood $B'$ of $0\in \mathfrak{g}$ such that $\exp_{G}|_{B'}$ is an analytic function, and since the exponential map of a Lie group is a homeomorphism on a sufficiently small ball around $0\in \mathfrak{g}$ to an open subset of $G$, we can ensure this property as well. Restricting $B'$ to a smaller ball $B$ makes $\exp_G|_B$ a restricted analytic function, and hence definable in $\mathbb{R}_{an,\exp}$.
\end{proof}
We are now ready to prove \Cref{prop:balls}.
\begin{proof}[Proof of \Cref{prop:balls}]
Let $0\in B\subset \mathfrak{g}$ be the open ball around $0$ such that $\exp_G:B\to G$ is a definable homeomorphism onto an open neighborhood of $\operatorname{id}_G$ guaranteed by \Cref{lem:7.1}. Take $N_h=\exp_{G}(B)h$. To show that this neighborhood works, let $\phi$ be the minimal complexity defining formula for $S$, and for $A\in B$ let $$\psi_{A,h}(t):\phi(\exp_G(At)h) \wedge 0\le t \le 1,$$
so that $$S_{\psi_{A,h}}:=\{t\in [0,1]:\exp_G(At)h\in S\}\subset [0,1].$$
We have $\operatorname{comp}(\psi_{A,h})=O_{G,\mathcal{F},\operatorname{comp}_{\mathcal{F}}(S)}(1)$ independent of the choice of $A$ and $h$, and if there is a size $m$ arithmetic progression $\{h,gh,\ldots,g^{m-1}h\}$ contained in $N_h$, taking $A\in B$ so that $\exp_G(\frac{i}{m-1}A)=g^{i}$ for $i=0,\ldots,m-1$, we see that $|S_{\psi_{A,h}}|\ge m$. However, because $S$ contains no exponential arcs, $S_{\psi_{A,h}}$
can't contain any intervals, and hence must be a finite collection of points. By \Cref{thm:unifbound}, $|S_{\psi_{A,h}}|=O_{G,\mathcal{F},\operatorname{comp}_{\mathcal{F}}(\psi)}(1)=O_{G,\mathcal{F},\operatorname{comp}_{\mathcal{F}}(S)}(1)$, independent of the choice of $A$, which is therefore an upper bound on the length of any arithmetic progression inside $N_h\cap S$ starting at $h$.
\end{proof}
\section{Bounding the set of bad translates}
\label{sec:selfirred}
Let $\mathbb{M}_\mathcal{F}$ be a fixed finitely generated o-minimal structure extending a dense linear order $(M,<)$, and $S\subset G$ a definable subset of a definable group. For notational convenience, we will take all constants to depend on $\mathcal{F}$ and $G$, and $\operatorname{comp}$ will be taken to be with respect to $\mathcal{F}$. In particular, we write $O_{\operatorname{comp}(S)}(1)$ instead of $O_{G,\mathcal{F},\operatorname{comp}_{\mathcal{F}}(S)}(1)$ throughout.

The following lemma gives us basic facts concerning the interaction between left-translates of $S$ and the notions of dimension and complexity.

\begin{lem}
\label{prop:basiccomplexity}
Let $\dim(S)=k$. Then
\begin{itemize}
    \item $gS$ is a definable set of dimension $k$ and $\operatorname{comp}(gS)=O_{\operatorname{comp}(S)}(1)$ independent of $g\in G$.
    \item $S\cap gS$ is definable, and $\operatorname{comp}(S\cap gS)=O_{\operatorname{comp}(S)}(1)$ independent of $g\in G$.
\end{itemize}
\end{lem}
\begin{proof}
Let $\phi(x)$ be a minimal complexity formula defining $S$. Then $\phi(g^{-1}x)$ defines $gS$, so $\operatorname{comp}(gS)=O_{\operatorname{comp}(S)}(1)$, and left multiplication by $g$ induces a definable bijection between $S$ and $gS$ so $\dim(gS)=\dim(S)=k$. The set $S\cap gS$ is definable by the formula $\phi(x)\wedge \phi(g^{-1}x)$, so $\operatorname{comp}(S\cap gS)=O_{\operatorname{comp}(S)}(1)$.
\end{proof}

Hence in particular $S\cap gS$ is definable, and the following definition makes sense.

\begin{defn}
Let $\operatorname{Bad}(S)\subset G$ be the subset $\{\operatorname{id}_G\ne g\in G: \dim(S\cap gS)=\dim(S)\}$.
\end{defn}
In this section, we prove the following.
\begin{prop}
\label{prop:boundedBadS}
If for all $h\in S$ there exists an open neighborhood $N_h$ of $h$ in $G$ such that $N_h\cap S$ does not contain unboundedly large arithmetic progressions starting at $h$, then $|\operatorname{Bad}(S)|=O_{\operatorname{comp}(S)}(1)$.
\end{prop}

\begin{lem}
\label{prop:BadSdef}
$\operatorname{Bad}(S)$ is definable of complexity $O_{\operatorname{comp}(S)}(1)$.
\end{lem}
\begin{proof}
Let $\phi(x)$ be a minimal complexity formula defining $S$. The subset
$$T=\{(g,x):\operatorname{id}_G\ne g\in G, x\in S\cap gS\}\subset G\times G$$
is defined by the formula $\psi(g,x): g\ne \operatorname{id}_G \wedge \phi(x)\wedge \phi(g^{-1}x)$, and hence $\operatorname{comp}(T)=O_{\operatorname{comp}(S)}(1)$. The projection $L:T\to G$ onto the first factor of $G$ has the property that $L|_S^{-1}(g)=S\cap gS$, so we conclude by applying \Cref{prop:definableloci} with $k=\dim(S)$.
\end{proof}
\begin{lem}
\label{prop:isometric}
If $g\in \operatorname{Bad}(S)$, then there exists $z\in S$, and an open neighborhood $z\in N\subset G$ such that
$$g(N\cap S)=gN\cap S.$$
\end{lem}
\begin{proof}
Let $X=S$ and $Y=g^{-1}S$. Because multiplication by $g$ is a definable bijection, $\dim(X)=\dim(Y)=k$ for some $k$.
Because $g\in \operatorname{Bad}(S)$ we have $\dim(S\cap g^{-1}S)=\dim(gS\cap S)=\dim(S)=k$, where the first equality follows since multiplication by $g$ is a definable bijection and the second equality follows by the definition of $\operatorname{Bad}(S)$. Hence applying \Cref{prop:XYk} gives the desired neighborhood $N$.
\end{proof}
We can now prove \Cref{prop:boundedBadS}. Before continuing, we give an informal description of the proof. Suppose that $\Bad(S)$ has dimension $\ge 1$. Let $B$ parametrize a basis of basic open sets $\{\gamma(F_b)\}_{b\in B}$ around $\operatorname{id}_G$. Then $\dim \Bad(S) \ge 1$ will imply that there are infinitely many triples $(h,b,g)\in S\times B \times G$ with $g\ne \operatorname{id}_G$ such that the neighborhood $h\gamma(F_b) \cap S$ around $h\in S$ is bijectively mapped to the neighborhood $gh\gamma(F_b)\cap S$ around $gh\in S$ through multiplication by $g$. Standard arguments let us find continuous paths $g(t),b(t),h(t)$, with $t$ in an interval $[x,y]$, of parameters satisfying this, with $g(t)$ injective. What amounts to a compactness argument shows we can take $b(t)$ to be a constant $b_0$ (we have to be careful since definable compactness is only considered in the context of o-minimal expansions of ordered groups, but fortunately we can use \Cref{fact:definablelift} and a monotonicity argument instead).

Now consider $t\in [x,y]$ very close to $x$. Multiplication by $g(x)$ bijectively maps $h(x)\gamma(F_{b_0})\cap S$ to $g(x)h(x)\gamma(F_{b_0})\cap S$, and multiplication by $g(t)^{-1}$ bijectively maps $g(t)h(t)\gamma(F_{b_0})\cap S$ to $h(t)\gamma(F_{b_0})$. Because there will be a large overlap between $h(x)\gamma(F_{b_0})$ and $h(t)\gamma(F_{b_0})$, and between $g(x)h(x)\gamma(F_{b_0})$ and $g(t)h(t)\gamma(F_{b_0})$, we will be able to find a slightly smaller open subset $U\subset h(x)\gamma(F_{b_0})$ such that $U\cap S$ is bijectively mapped to $g(x)U\cap S$ (this would happen for any subset), and $g(x)U\subset g(t)h(t)\gamma(F_{b_0})$ so that $g(x)U\cap S$ is bijectively mapped to $g(t)^{-1}g(x)U\cap S$ through multiplication by $g(t)^{-1}$. Since we will be able to take $U$ independent of $t$ provided $t$ is sufficiently close to $x$, we have a continuous path $g(t)^{-1}g(x)$ starting at $\operatorname{id}_G$ such that $U\cap S$ is bijectively mapped to $g(t)^{-1}g(x)U$ through multiplication by $g(t)^{-1}g(x)$. For $t$ sufficiently close to $x$ this implies that for any $m\ge 1$ the elements $h(x),g(t)^{-1}g(x)h(x),\ldots, (g(t)^{-1}g(x))^{m-1}h(x)$ all lie in $S$, which for $m$ sufficiently large violates the assumption that there is a neighborhood $N_{h(x)}$ of $h(x)$ such that $N_{h(x)}\cap S$ does not contain arbitrarily large arithmetic progressions starting at $h(x)$ unless $g(t)^{-1}g(x)$ is torsion of order $\le m$. But for $t$ sufficiently close to $x$ this can't happen by the ``no small torsion'' property \eqref{tag:NST}.
\begin{proof}[Proof of \Cref{prop:boundedBadS}]
Throughout the proof, the topology on any definable group or subset of a definable group will always be the ``t-topology'' (see \Cref{rem:Pillay}). Also we will frequently use without comment the fact that if a subset $N\subset G$ has $g(N\cap S)=gN\cap S$, and $N'\subset N$, then $g(N'\cap S)=gN'\cap S$.

If $\dim \operatorname{Bad}(S) =0$, then by \Cref{prop:BadSdef} we have $\operatorname{comp}(\operatorname{Bad}(S))=O_{\operatorname{comp}(S)}(1)$, so by \Cref{thm:unifbound} we have $|\operatorname{Bad}(S)|=O_{\operatorname{comp}(S)}(1)$. Hence it suffices to show that we do not have $\dim \operatorname{Bad}(S)\ge 1$.

Suppose by way of contradiction that $\dim \operatorname{Bad}(S) \ge 1$. Let $\gamma:\prod_{i=1}^{\dim(G)} (c_i,d_i)\to G$ be a homeomorphism to an open neighborhood around $\operatorname{id}_G\in G$, and let $z=(z_1,\ldots,z_{\dim(G)})\in \prod (c_i,d_i)$ be the point with $\gamma(z)=\operatorname{id}_G$. Let $$B_i=\{(c'_i,d'_i): c_i<c'_i<z_i<d'_i<d_i\}=(c_i,z_i)\times (z_i,d_i)\subset M^2,$$ and let $B=\prod B_i$. For $b=((c'_1,d'_1),\ldots, (c'_{\dim(G)},d'_{\dim(G)}))\in \prod B_i$ denote $F_{b}$ for the product of open intervals $\prod_{i=1}^{\dim(G)}(c'_i,d'_i)\subset M^{\dim(G)}$.

Consider the definable set $$\Lambda=\{(h,b,g)\in S\times B \times G:g\ne \operatorname{id}_G, \text{ and }h,gh\in S,\text{ and }g(h\gamma(F_{b})\cap S)=gh\gamma(F_{b})\cap S\}.$$
\begin{clm}
There is a continuous definable map $(h(t),b_0,g(t)):[x,y]\to \Lambda$, with $b_0\in B$ a constant, $g(t)$ injective, and $[x,y]\subset M$ an interval.
\end{clm}
\begin{proof}
Because $\{h\gamma(F_{b}):b\in B\}$ is a basis of neighborhoods of $h\in G$, by \Cref{prop:isometric} the projection $\Lambda \to G$ has image containing $\Bad(S)$, and since $\dim \Bad(S) \ge 1$ one of the coordinate projections of the image of $\Lambda$ in $G$ to $M$ contains an interval $I'$. Denote the composite projection $\phi:\Lambda \to M$ and the inclusion $g'(t):I'\to \phi(\Lambda)$. Note that $\Lambda\subset S\times B \times G$, which can be embedded into the definable group $G\times G\times G \times G$ by including $B\to G\times G$ via the map $((c_1',d_1'),\ldots,(c_{\dim(G)}',d_{\dim(G)}'))\mapsto (\gamma(c_1',\ldots,c_{\dim(G)}'),\gamma(d_1',\ldots,d_{\dim(G)}'))$. Therefore we can apply \Cref{fact:definablelift} to obtain a lift $(h(t),b(t),g(t)):I'\to \Lambda$ of $g'(t)$ along $\phi$, and $g(t)$ is injective since $g'(t)$ is injective. By \Cref{fact:contshrink} we can find a subinterval $[x,y]\subset I'$ so that this map is continuous, and the $2\dim(G)$ components of the map $b(t)=((c'_1(t),d'_1(t)),\ldots,(c'_{\dim(G)}(t),d'_{\dim(G)}(t)))$ are monotone. Let $$b_0=((\max c'_1(t),\min d'_1(t)),\ldots,(\max c'_{\dim(G)}(t),\min d'_{\dim(G)}(t)))$$ (which exists since the extremes of $c'_i(t)$ and $d'_i(t)$ are attained at $x$ and $y$). Then because $h(x)\gamma(F_{b_0})\subset h(x)\gamma(F_{b(t)})$ for all $t$, we have $(h(t),b_0,g(t))\in \Lambda$ for all $t\in I$.
\end{proof}

\begin{clm}
\label{clm:existsU}
There is a sub-interval $[x,y']\subset [x,y]$ and an open neighborhood
$h(x)\in U$ such that $g(t)^{-1}g(x)(U\cap S)=g(t)^{-1}g(x)U\cap S$ for $t\in [x,y']$.
\end{clm}
\begin{proof}
Consider the open subset $V\subset [x,y]\times F_{b_0}$ of pairs $(t,w)$ such that $(g(t)h(t))^{-1}g(x)h(x)\gamma(w)\in \gamma(F_{b_0})$. Then since boxes form a basis of the topology of $[x,y]\times F_{b_0}$, and $V$ contains some point of the form $(x,w)$ (in fact it contains every point of the form $(x,w)$), we can find a box $[x,y']\times F_{b_0'}$ such that $[x,y']\times F_{b_0'}\subset V$, and we can take $U=h(x)\gamma(F_{b_0'})$. Then by construction, $g(x)U\subset g(t)h(t)\gamma(F_{b_0})$ for all $t\in [x,y']$.
Because $U\subset h(x)\gamma(F_{b_0})$, we have $g(x)(U\cap S)=g(x)U\cap S$. Because $g(t)^{-1}(W\cap S)=g(t)^{-1}W\cap S$ for $W=g(t)h(t)\gamma(F_{b_0})$ and $g(x)U\subset g(t)h(t)\gamma(F_{b_0})$, we have $g(t)^{-1}(g(x)U\cap S)=g(t)^{-1}g(x)U \cap S$. Hence
$$g(t)^{-1}g(x)(U\cap S)=g(t)^{-1}(g(x)U\cap S)=g(t)^{-1}g(x)U \cap S.$$
\end{proof}
Returning to the proof of \Cref{prop:boundedBadS}, consider the continuous function \begin{align*}a_m(t):[x,y']&\to G^m\\
t&\mapsto (h(x),(g(t)^{-1}g(x))h(x),\ldots,(g(t)^{-1}g(x))^{m-1}h(x)).\end{align*}
The following claim concerning $a_m(t)$ will allow us to construct unboundedly large arithmetic progressions in $N_{h(x)}\cap S$ starting at $h(x)$, contradicting the definition of $N_{h(x)}$.
\begin{clm}There is an open neighborhood $x\in \Gamma_m\subset [x,y']$ such that $a_m(\Gamma_m)\subset (N_{h(x)}\cap S)^m$.
\end{clm}
\begin{proof}
Let $U$ be as in \Cref{clm:existsU}.
Because $a_m(x)\in (N_{h(x)}\cap U)^m$, by continuity we can find an open neighborhood $x\in \Gamma_m\subset [x,y']$ such that $a_m(t)\subset (N_{h(x)}\cap U)^m$ for all $t\in [x,y']$. We claim that this implies $a_m(t)\in S^m$ as well.  Indeed, let $k\in [0,m-1]$ be the largest integer such that $(g(t)^{-1}g(x))^ih(x)\in S$ for $0\le i \le k$ (this is well-defined since $h(x)\in S$). If $k<m-1$, then since $(g(t)^{-1}g(x))^kh(x)\in U\cap S$, we have
$$(g(t)^{-1}g(x))^{k+1}h(x)\in (g(t)^{-1}g(x))(U\cap S)=g(t)^{-1}g(x)U\cap S\subset S,$$
contradicting the maximality of $k$. Hence $k=m-1$, which shows $a_m(t)\in S^m$ as desired.
\end{proof}
Returning to the proof of \Cref{prop:boundedBadS}, for any $t\in \Gamma_m$, the coordinates of $a_m(t)$ exhibits a length $m$ arithmetic progression lying in $N_{h(x)}\cap S$ provided $g(t)^{-1}g(x)$ is not torsion of order $\le m$. Note that because $g$ is injective, we have $g(t)^{-1}g(x)\ne \operatorname{id}_G$ for any $x\ne t\in \Gamma_m$.

Because $G$ has the ``no small torsion'' property (\ref{tag:NST}) by \Cref{thm:NST}, there exists an open neighborhood $\operatorname{id}_G\in \widehat{N}_m\subset G$ such that $\widehat{N}_m\setminus \{\operatorname{id}_G\}$ contains no torsion of order $\le m$. By continuity, there is a subinterval $x\in I_m\subset \Gamma_m$ such that $g(t)^{-1}g(x)\in \widehat{N}_m$ for $t\in I_m$ and since $g(t)$ is injective, $g(t)^{-1}g(x)\in \widehat{N}_m\setminus \{\operatorname{id}_G\}$ for $x\ne t\in I_m$. Hence for $x\ne t\in I_m$ we have $g(t)^{-1}g(x)$ is not torsion of order $\le m$, so $a_m(t)$ exhibits a length $m$ arithmetic progression in $N_{h(x)}\cap S$ starting at $h(x)$. Since this is true for all $m$, we see that $N_{h(x)}\cap S$ contains unboundedly large arithmetic progressions starting at $h(x)$, which contradicts the definition of $N_{h(x)}$.
\end{proof}
\section{Bounding the size of arithmetic progressions}
We work with the same conventions as in \Cref{sec:selfirred}. In this section we prove the following proposition, showing that subsets $S\subset G$ which locally avoid unboundedly large arithmetic progressions also avoid unboundedly large arithmetic progressions globally.

\begin{prop}
\label{prop:boundedAP}
Suppose for all $h\in S$ there exists an open neighborhood $N_h$ which does not contain unboundedly large arithmetic progressions starting at $h$. Then there exists an $r=O_{\operatorname{comp}(S)}(1)$ such that for all $g\in G$, either $g$ is torsion of order $\le r$, or $\{h,gh,\ldots,g^{r-1}h\}\not \subset S$ for all $h\in S$. In particular, $S$ does not contain size $r$ arithmetic progressions.
\end{prop}
\begin{proof}
We induct on the dimension $k$ of $S$. For $k=0$, $|S|=O_{\operatorname{comp}(S)}(1)$ by \Cref{thm:unifbound} and the result is trivial.

Now, assume $k>0$ and the result is true for all smaller $k$. By \Cref{prop:basiccomplexity}, for all $x\in G$ we have $\operatorname{comp}(S\cap xS)=O_{\operatorname{comp}(S)}(1)$, so by the inductive hypothesis there exists an $r'=O_{\operatorname{comp}(S)}(1)$ such that if $\dim(S\cap xS)\le k-1$, then $\{h,gh,\ldots,g^{r'-1}h\}\not\subset S\cap xS$ whenever $g$ is not torsion of order $\le r'$.

By \Cref{prop:boundedBadS}, there exists an $m=O_{\operatorname{comp}(S)}(1)$ with $|\operatorname{Bad}(S)|\le m$.

We claim that for $r=m+r'$, if $g$ is not torsion of order $\le r$ then we have $\{h,gh,\ldots,g^{r-1}h\}\not\subset S$. Indeed, suppose that $\{h,gh,\ldots,g^{r-1}h\}\subset S$. As $|\operatorname{Bad}(S)|\le m$ and $g$ is not torsion of order $\le m+1$, there exists $i\in \{1,\ldots,m+1\}$ such that $g^i \not \in \operatorname{Bad}(S)$. Hence $\dim(S\cap g^{-i}S)=\dim(g^iS\cap S)\le k-1$ since multiplication by $g^i$ is a definable bijection, so $\{x,gx,\ldots,g^{r'-1}x\}\not\subset S\cap g^{-i}S$. But this is false since we have both $\{x,gx,\ldots,g^{r'-1}x\}\subset S$ and $\{g^ix,g^{i+1}x,\ldots,g^{r'-1+i}x\}\subset S$.
\end{proof}

\section{Proof of \Cref{thm:ominimaldimension}}
We work with the same conventions as in \Cref{sec:selfirred}. In this section we prove \Cref{thm:ominimaldimension}. We have to first prove the following result concerning ``self-irreducible'' sets.
\begin{defn}
Say that $S$ is self-irreducible if for all $\operatorname{id}_G\ne g\in G$ we have $\dim(S\cap gS)<\dim(S)$.
\end{defn}
\begin{prop}
\label{prop:breakup}
If for all $h\in S$ there exists an open neighborhood $N_h$ which does not contain unboundedly large arithmetic progressions starting at $h$, then there exists $C=O_{\operatorname{comp}(S)}(1)$ such that we can write $S=S_1\cup \ldots \cup S_C$ with $S_i$  definable and self-irreducible, and $\operatorname{comp}(S_i)=O_{\operatorname{comp}(S)}(1)$ for all $i$.
\end{prop}
\begin{proof}
Induct on the dimension of $S$. For $\dim(S)=0$ we have $|S|=O_{\operatorname{comp}(S)}(1)$ by \Cref{thm:unifbound}, so we can decompose $S$ into its individual points. Assume now $\dim(S)>0$ and the result is true for all smaller dimensions. By \Cref{prop:boundedBadS} we have  $|\operatorname{Bad}(S)|=O_{\operatorname{comp}(S)}(1)$. For $g\in \operatorname{Bad}(S)$, we claim we can write
$S=S_1'\cup \ldots \cup S_m'$ with $m=O_{\operatorname{comp}(S)}(1)$ and $\operatorname{comp}(S_i')=O_{\operatorname{comp}(S)}(1)$ for all $i$, such that $S_i'\cap gS_i'=\emptyset$. Given such a decomposition, for each $i$ we either have $S_i'$ is of lower dimension than $S$, in which case we can apply the inductive hypothesis to $S_i'$, or else $g\not\in \operatorname{Bad}(S_i')\subset \operatorname{Bad}(S)$, so $|\Bad(S_i')|<|\Bad(S)|$ and we can apply this same reasoning recursively on each $S_i'$.

To produce such a decomposition, first note that by \Cref{prop:boundedAP}, there exists an $m=O_{\operatorname{comp}(S)}(1)$ such that for all $\operatorname{id}_G\ne g\in G$, either $g$ is torsion of order $\le m$, or else $\{h,gh,\ldots,g^{m-1}h\}\not\subset S$ for any $h$.

We now have two cases. If $g$ is torsion of order $\le m$, then under some definable ordering of $G$, let $S_i'$ be the set of $x\in S$ such that $x$ is $i$'th smallest in the set $\{x,gx,\ldots,g^{m-1}x\}\cap S$. Otherwise $g$ is not torsion of order $\le m$, and we let $S_i'$ be those $x\in S$ such that $\{x,gx,\ldots,g^{i-1}x\}\subset S$ but $g^ix\not\in S$.
\end{proof}

Now we are ready to prove \Cref{thm:ominimaldimension}.
\begin{proof}[Proof of \Cref{thm:ominimaldimension}]
First, we show that $\dimrob(S)\le \dimself(S)$. If $\dimself(S)=\infty$ then there is nothing to prove, so assume that $\dimself(S)=k<\infty$. Then we have to show that $\mathcal{H}_{S,k+1}$ does not contain grids $A_1\times \ldots \times A_{k+1}$ with $|A_1|=\ldots = |A_{k+1}|=C$ unboundedly large. But if $X_i$ is a uniform random variable on $A_{i+1}$, by \Cref{thm:selfirreduciblecomb} we have $\mathbb{P}(\prod_{i=0}^k X_i \in S)=O_S(C^{-\frac{1}{2^k}})$, so for $C$ sufficiently large this probabiliy is less than $1$. Consequently there is at least one element of $A_1\times \ldots \times A_{k+1}$ which does not lie in $\mathcal{H}_{S,k+1}$. Hence $\dim_{rob}(S)\le \dim_{self}(S)$ as desired.

It remains to show that if $S$ does not contain unboundedly large arithmetic progressions, then $S$ is of self-translate dimension $\le \dim(S)$ of complexity $O_{\operatorname{comp}(S)}(1)$. We induct on $\dim(S)$. For $\dim(S)=0$, $S$ is a finite set of points so by \Cref{thm:unifbound} we have $|S|=O_{\operatorname{comp}(S)}(1)$ and we can simply decompose $S$ into its individual points. Assume therefore that $k=\dim(S)>0$ and the result is true for all lower dimensions.

By \Cref{prop:breakup}, there exists $C'=O_{\operatorname{comp}(S)}(1)$ such that $S=S_1\cup \ldots \cup S_{C'}$ with each $S_i$ definable and self-irreducible, and  $\operatorname{comp}(S_i)=O_{\operatorname{comp}(S)}(1)$. For $\operatorname{id}_G\ne g\in G$, we therefore have by \Cref{prop:basiccomplexity} that $\operatorname{comp}(S_i\cap gS_{i})=O_{\operatorname{comp}(S)}(1)$, and $\dim(S_i\cap gS_{i})\le k-1$ by definition of self-irreducibility. Hence we may apply the inductive hypothesis to deduce the existence of a constant $C''=O_{\operatorname{comp}(S)}(1)$ such that $S_i\cap gS_{i}$ is of self-translate dimension $\le k-1$ of complexity $C''$ for all $i,g$. We may then take $C=\max(C',C'')$.
\end{proof}

\section{Proofs of \Cref{thm:trajectory} and \Cref{thm:generalominimal}}

\begin{proof}[Proof of \Cref{thm:generalominimal}]
First we show part a).
If $S$ contains unboundedly large arithmetic progressions, then $\dimrob(S)=\infty$, and \Cref{prop:boundedAP} shows there is an $s\in S$ such that every open neighborhood of $s$ contains unboundedly large arithmetic progressions starting at $s$. Conversely, if $S$ does not contain unboundedly large arithmetic progressions, then by \Cref{thm:ominimaldimension} we have $\dimrob(S)\le \dim(S)<\infty$. The second part concerning definable families follows immediately from \Cref{prop:boundedAP}.

Next we show part b). Let $r=\dimrob(S)$. If for some $r'$ there is a definable map $\Gamma(t_1,\ldots,t_{r'})=\gamma_1(t_1)\ldots\gamma_{r'}(t_{r'}):\prod_{i=1}^{r'} [a_i,b_i] \to S$ with $\gamma_i:[a_i,b_i]\to G$ injective, then $\gamma_1([a_1,b_1])\times \ldots\times \gamma_{r'}([a_{r'},b_{r'}])\subset \mathcal{H}_{S,r'}$, so $r'\le \dimrob(S)$. Conversely, \Cref{cor:Chernikov} shows that there are curves $\mathcal{C}_1,\ldots, \mathcal{C}_r\subset G$ such that for $\mathcal{C}_1\cdots \mathcal{C}_r$ the pointwise product set we $\mathcal{C}_1\cdots \mathcal{C}_r\subset S$.

We claim that $\dim(\mathcal{C}_1\cdots \mathcal{C}_r)=r$. Indeed, note on the one hand that $\mathcal{C}_1\cdots \mathcal{C}_r$ is the image of $\mathcal{C}_1\times \cdots \times \mathcal{C}_r$ under the multiplication map $\Phi:\mathcal{C}_1\times \cdots \times \mathcal{C}_r\to \mathcal{C}_1\cdots \mathcal{C}_r$, so $\dim(\mathcal{C}_1\cdots\mathcal{C}_r)\le r$. On the other hand, $\mathcal{C}_1\times \cdots \times \mathcal{C}_r \subset \mathcal{H}_{\mathcal{C}_1\cdots \mathcal{C}_r,r}$, so $\dimrob(\mathcal{C}_1\cdots \mathcal{C}_r)\ge r$, and as $\mathcal{C}_1\cdots \mathcal{C}_r\subset S$ does not contain unboundedly large arithmetic progressions, by \Cref{thm:ominimaldimension} we have $\dim(\mathcal{C}_1\cdots \mathcal{C}_r) \ge \dimrob(\mathcal{C}_1\cdots \mathcal{C}_r)\ge r$. 

For each of the curves $\mathcal{C}_1,\ldots,\mathcal{C}_r$, there is a coordinate projection to $M$ whose image has dimension $1$, and in particular contains an interval. Hence by \Cref{fact:definablelift} and \Cref{fact:contshrink}, we may assume (after shrinking the intervals and curves) that there are continuous definable bijections $\gamma_i:[a_i,b_i]\to \mathcal{C}_i$. Note that $\Phi\circ(\gamma_1,\ldots,\gamma_r)$ is a surjective map $\prod_{i=1}^r [a_i,b_i]\to \mathcal{C}_1\cdots \mathcal{C}_r$, and both the domain and codomain have dimension $r$. By \Cref{fact:definablelift} applied to the group $G^r$, there is a definable section of $\Phi$, and hence of $\Phi\circ(\gamma_1,\ldots,\gamma_r)$. The image of this latter section has dimension $r$, so by \Cref{prop:XYk} contains an open subset of $\prod_{i=1}^r [a_i,b_i]$. Restricting the intervals further so that their product lies in this open subset, we may assume that $\Phi\circ(\gamma_1,\ldots,\gamma_r)$ is injective. Finally, by \Cref{fact:contshrink} we may restrict the intervals one last time to ensure continuity and injectivity of $\gamma_1,\ldots,\gamma_r$.
\end{proof}

\begin{proof}[Proof of \Cref{thm:trajectory}]
By combining \Cref{prop:balls} and \Cref{prop:boundedAP}, if $S\subset G$ is a definable subset of a definable group in $\mathbb{R}_{an,\exp}$, then $S$ does not contain an exponential arc if and only if $S$ avoids unboundedly large arithmetic progressions. Hence everything follows from \Cref{thm:generalominimal} except the analyticity of $\gamma_1,\ldots,\gamma_n$ and $\Gamma$, which we can guarantee by shrinking the intervals $[a_i,b_i]$ since by van den Dries and Miller \cite[Corollary 6.12, (ii)]{DM94} and \cite[Theorem 8.8, (II)]{DM94}, every $\mathbb{R}_{an,\exp}$-definable function is analytic on some open set.
\end{proof}

\bibliographystyle{plain}
\bibliography{bib}

\end{document}